\documentclass[11pt]{amsart} 
\usepackage{amssymb,amsmath,latexsym,enumerate,graphicx,bbm,mathptmx,cite,ifthen}
\usepackage{cmbright}
\usepackage{bbm,mathtools}

\usepackage{microtype} 

\usepackage[dvipsnames]{xcolor}
\allowdisplaybreaks

\newtheorem{theorem}{Theorem}  [section] \numberwithin{equation}{section}
\newtheorem{corollary}[theorem]{Corollary}
\newtheorem{proposition}[theorem]{Proposition}
\newtheorem{lemma}[theorem]{Lemma}
\newtheorem{conjecture}[theorem]{Conjecture}
\newtheorem{exam}{Example}
\newenvironment{example}{\begin{exam}\sf}{\end{exam}}
\newtheorem*{rem}{Remarks}
\newenvironment{remark}{\begin{rem}\sf}{\end{rem}}

\usepackage{rotating}
\usepackage{tikz}
\usepackage{tikz-3dplot}
 \usetikzlibrary{decorations.markings}
\tikzstyle{vertex}=[circle, draw, inner sep=0pt, minimum size=4pt]
\newcommand{\vertex}{\node[vertex]}
\tikzset{mid arrow/.style={postaction={decorate,very thick, decoration={
        markings,
        mark=at position 0.6 with {\arrow{#1}}
    }}}}

\renewcommand\emptyset{\varnothing}

\newcommand\commentout[1]{}
\newcommand\Def[1]{{\bf #1}}

\newcommand\ehr{\operatorname{ehr}}

\newcommand{\rk}[1][\pi]{p_{#1}}
\newcommand{\maj}{\mathrm{maj}}
\newcommand{\des}{\mathrm{Des}}
\newcommand{\asc}{\mathrm{Asc}}
\newcommand{\desnum}{\mathrm{des}}
\newcommand{\ascnum}{\mathrm{asc}}

\newcommand{\rev}[1]{#1^{\mathrm{rev}}}
\newcommand{\op}[1]{#1^{\mathrm{op}}}
\newcommand{\order}[1]{\mathcal{O}(#1)}
\newcommand{\acyclic}[1]{\mathcal{A}(#1)}
\newcommand{\jhset}[1]{\mathcal{L}(#1)}
\newcommand{\ps}{\textbf{ps}}

\newcommand{\qbinom}[2]{\genfrac[]{0pt}{}{#1}{#2}_q}
\DeclarePairedDelimiter\ceil{\lceil}{\rceil}
\DeclarePairedDelimiter\floor{\lfloor}{\rfloor}

\makeatletter 
\newtheorem*{rep@theorem}{\rep@title}\newcommand{\newreptheorem}[2]{%
\newenvironment{rep#1}[1]{%
\def\rep@title{\bf #2 \ref{##1}}%
\begin{rep@theorem}}%
{\end{rep@theorem}}}
\makeatother
\newreptheorem{theorem}{Theorem}

\newcounter{teach}
\begin{document}

\title{Generating functions of $q$-chromatic polynomials}

\author{Matthias Beck}
\address{Department of Mathematics\\
          San Francisco State University\\
          San Francisco, CA 94132}
 \email{mattbeck@sfsu.edu}
 \urladdr{https://matthbeck.github.io/}

 \author{Benjamin Braun}
 \address{Department of Mathematics\\
          University of Kentucky\\
         Lexington, KY 40506}
\email{benjamin.braun@uky.edu}
 \urladdr{https://sites.google.com/view/braunmath/}

 \author{Alvaro Cornejo}
 \address{Department of Mathematics\\
          University of Kentucky\\
         Lexington, KY 40506}
\email{alvaro.cornejo@uky.edu}
 \urladdr{https://sites.google.com/view/alvaro-cornejo/home}





\date{25 September 2025}

\begin{abstract}
    Given a graph $G=(V,E)$ and a linear form $\lambda \in \mathbb{Z}_{ > 0 }^V$, Bajo et al.~(2025) introduced the
$q$-chromatic polynomial $\chi_G^\lambda(q,n) := \sum q^{\sum_{v \in V} \lambda_v c(v)}$ where the sum is over all
proper colorings $c: V \to [n] := \{ 1, 2, \dots, n \}$; they showed that $\chi_G^\lambda(q,n)$ is a polynomial
in $[n]_q := 1 + q + \dots + q^{ n-1 } $ with coefficients in $\mathbb{Z}(q)$.
    For $d \in \mathbb{Z}_{>0}$ and the linear form given by $(d,d^2,\ldots,d^d)$, we show that the $q$-chromatic polynomial distinguishes labeled graphs with vertex set $[d]$.
    Using permutation statistics introduced by Chung--Graham (1995), called $G$-statistics, and polyhedral geometry, we give the multivariate integer point transform for the region of proper
colorings of a given graph~$G$.
    This integer point transform allows us to find the generating function for the $q$-chromatic polynomial with respect to any linear form.
    We further specialize these results to the linear form $\mathbbm{1} := (1, 1, \dots, 1)$, which allows us to write the $q$-chromatic polynomial in the $q$-binomial basis, clarifying expressions found by Bajo et al.
    Moreover, we show that $G$-statistics are compatible with the theory of order polytopes used by Bajo et
al.\ and Chow (1999).
    This yields further properties for the generating function of $q$-chromatic polynomial with linear form $\mathbbm{1}$, where certain coefficients of the numerator polynomial are palindromic polynomials in $q$.
\end{abstract}

\maketitle



\section{Introduction}

For a graph $G=(V,E)$ and a linear form $\lambda \in \mathbb{Z}_{ >0 }^V$, Bajo et al.~\cite{qchromatic}
introduced a $q$-analogue of the chromatic polynomial,
\[
    \chi_G^\lambda(q,n) := \sum_{\substack{\text{proper colorings} \\ c:V \to [n]}} q^{\sum_{v \in V} \lambda_v c(v)},
\]
which we call the \Def{$q$-chromatic polynomial} of $G$, due to the fact that $\chi_G^\lambda(q,n)$ is a polynomial
in $[n]_q := 1 + q + \dots + q^{ n-1 } $ with coefficients in $\mathbb{Z}(q)$~\cite{qchromatic}.
For $V=\{v_1,\ldots,v_d\}$ and linear form $\mathbbm{1} := (1, 1, \dots, 1)$, the $q$-chromatic polynomial is known
to be the \Def{principal specialization} $X_G(q,q^2,\hdots,q^n,0,0,\hdots)$ of Stanley's \Def{chromatic symmetric function}~\cite{stanleychromatic} 
\[
    X_G(x_1,x_2,\ldots) := \sum_{\substack{\text{proper colorings} \\ c:V \to \mathbb{Z}_{>0}}} x_{c(v_1)} x_{c(v_2)} \cdots x_{c(v_d)}.
\]
Stanley conjectured that the chromatic symmetric function distinguish trees; Loehr and
Warrington~\cite{loehrwarrington} conjectured, more strongly, that the principal specialization $X_G(q,q^2,\hdots,q^n,0,0,\hdots)$ distinguishes trees. 
Since the principal specialization is the $q$-chromatic polynomial with the particular linear form
$\mathbbm{1}$, Bajo et al.\ asked if there exists some linear form $\lambda \in \mathbb{Z}_{>0}^d$ such
that the $q$-chromatic polynomial distinguishes trees. In Theorem~\ref{thm:distinguishedlabeled}, we answer this question in the affirmative (for all graphs, not
just for trees), for
$\lambda = (d, d^2, \ldots, d^d)$ where $d = |V|$.

In Section~\ref{sec:gstats}, we express the generating function of the $q$-chromatic
polynomial of $G$ for any linear form $\lambda \in \mathbb{Z}_{ >0 }^V$ using polyhedral geometry and certain statistics on permutations and graphs. 
$G$-descents were introduced by Chung--Graham~\cite{chunggrahamcoverpol} in order to give a combinatorial
interpretation of the coefficients of the chromatic polynomial $\chi_G(n)$ when expressed in the basis
$\left\{\binom{n+k}{d} : 0 \leq k \leq d\right\}$.
Their proof was not combinatorial, however, such a proof was provided by Steingrímsson~\cite{steingrimssoncoloring}.
Using these ideas in conjunction with polyhedral geometry, we express a multivariate generating function of
graph colorings (Theorem~\ref{thm:multivariablechromaticgenerating} below). 
Consequently, we prove an expression for the generating function for the $q$-chromatic polynomial for any
linear form (Theorem~\ref{thm:genfuncweighted}).

In Section~\ref{sec:linearform1}, we apply our general expression of the $q$-chromatic generating function to the linear form $\mathbbm{1}\in
\mathbb{Z}^d$.
In particular, we give a rational expression of the generating function of $\chi_G^\mathbbm{1}(q,n)$ for any graph $G$ in Theorem~\ref{thm:qchromaticallones}.
From this theorem, we give a method to compute the coefficients of $\chi_G^\mathbbm{1}(q,n)$ when written in the basis
$\left\{\qbinom{n+k}{d} : 0 \leq k \leq d\right\}$ using $G$-statistics, yielding a $q$-analogue of Chung--Graham's expression of the chromatic polynomial in the binomial basis.
Moreover, Chow~\cite{chow1999} writes the chromatic symmetric function in terms of fundamental quasisymmetric
functions with $G$-statistics, and the generating function for the principal specialization of Chow's expression coincides with our generating function for the $q$-chromatic polynomial with linear form $\mathbbm{1}$, which we prove in Section~\ref{subsec:symmetricfunctions}.

In Section~\ref{sec:gstatsandorientations}, emulating work by Chow, we give a bijection
 between permutations and pairs of acyclic orientations and linear extensions of the induced poset (Theorem~\ref{thm:bijectiongstats}).
In addition, Bajo et al.\ show the $q$-chromatic polynomial can be expressed as a sum of $q$-Ehrhart polynomials of order polytopes.
Using their result, in Theorem~\ref{thm:qchromaticgenfunctiondoublesum} we express the generating function of the $q$-chromatic polynomial with linear form $\mathbbm{1}$ using pairs of acyclic orientations and linear extensions of induced posets which agrees with our previous method using $G$-statistics.
We use our bijection and different expressions of the generating function of the $q$-chromatic polynomial to prove the following:
\begin{enumerate}
    \item In Section~\ref{subsec:gmajorpoly}, we write the leading term of
$\widetilde{\chi}_G^\mathbbm{1}(q,x)$ as a single sum relating to the $G$-major index polynomial
(Corollary~\ref{cor:leadingcoef}). 
    Moreover, we show that the degree of the $G$-major index polynomial can be computed by minimizing $\sum_{v \in V} c(v)$ over all proper colorings $c$ in Corollary~\ref{cor:gmajordegree}.
    In particular, we give bounds for the degree of the $G$-major index polynomial for trees 
(Corollary~\ref{cor:treegmajordegree}).

    \item In Corollary~\ref{cor:palindromic} we prove that the coefficients of $\chi_G^\mathbbm{1}(q,n)$ when written in the basis  \\$\left\{\qbinom{n+k}{d} : 0 \leq k \leq d\right\}$
    are palindromic polynomials, up to a shift.
\end{enumerate} 


\section{The linear form $\lambda=(k,k^2,\ldots,k^d)$}\label{sec:genericlinearform}

Let $G=([d],E)$ be a graph, let $k \in \mathbb{Z}_{>0}$ such that $k \geq d$, and consider the linear form
$\lambda=(k,k^2,\ldots,k^d)$. We will now show that $\chi_G^\lambda(q,n)$ is unique for every graph $G$. 
The key tool will be the unique base-$k$ representation of a natural number.


\begin{theorem}\label{thm:distinguishedlabeled}
    For $d \in \mathbb{Z}_{>0}$, let $k \in \mathbb{Z}_{>0}$ such that $k \geq d$ and let $\lambda_i := k^i$. 
    If $G$ and $H$ are graphs on $[d]$ such that $\chi^\lambda_G(q,d-1) = \chi^\lambda_H(q,d-1)$, then $G=H$.
\end{theorem}
\begin{proof}
    Suppose $\chi^\lambda_G(q,d-1) = \chi^\lambda_H(q,d-1)$, that is,
    \[
        \sum_{\substack{\text{proper } G-\text{coloring} \\ c:[d] \to [d-1]}} q^{\sum_{i \in [d]} \lambda_i
c(i)} \, = \sum_{\substack{\text{proper } H-\text{coloring} \\ f:[d] \to [d-1]}} q^{\sum_{i \in [d]} \lambda_i f(i)}.
    \]
By the uniqueness of the base-$k$ representation of a natural number, each proper coloring of $G$ gives a unique
power of $q$ on the left-hand side and, similarly, each proper coloring of $H$ gives unique powers of $q$ on the right-hand side. 
    In order for equality to hold, the powers appearing must be the same and hence every proper coloring of $G$ using at most $d-1$ colors is also a proper coloring of $H$ and vice versa. 
    In other words, the two graphs share the same $(d-1)$-colorings.
    
    In particular, some of the $(d-1)$-colorings encode the edges of a graph since we can use them to test if we
have an edge between any two vertices, as follows.
    Suppose $v,w \in [d]$ are vertices of some graph on $[d]$ and consider the coloring $f: [d] \to [d-1]$
defined by $f(v) = f(w)=1$ and each other vertex being assigned a unique color between $2$ and $d-1$.
    If $v,w$ form an edge, $f$ will not be a proper coloring, otherwise it will be.
    Hence if the power of $q$ corresponding to $f$ appears we do not have the edge $v,w$ and it does not appear we do have an edge.
    Since these $(d-1)$-colorings of $G$ and $H$ must coincide, $G$ and $H$ have the same edge sets, showing $G=H$.
\end{proof}

Theorem~\ref{thm:distinguishedlabeled} shows that it can be useful to consider various linear forms aside from the all-ones vector~$\mathbbm{1}$.


\section{$G$-statistics and the regions of proper colorings}\label{sec:gstats}


\subsection{$G$-statistics}
Let $G = (V, E)$ be a graph with $V = [d]$. 
The following definition was introduced by Chung--Graham \cite{chunggrahamcoverpol}.
Given a permutation $\pi = [\pi(1) \cdots \pi(d)] \in S_d$ written in one-line notation, define the \Def{rank} of $\pi(i)$ with respect to $\pi$, denoted $\rk(\pi(i))$, to be the largest $r \in \mathbb{Z}_{>0}$ for which there
exist positions $i_1 < i_2 < \dots < i_r = i$ such that $\{\pi(i_j), \pi(i_{j+1}) \} \in E$ for $1 \le j < r$. 
This gives a function $\rk:[d] \to [d]$ which we call the \Def{rank function of $\pi$}.
Steingrímsson used a variation of rank to instead respect an ordered set partition of the vertex set \cite{steingrimssoncoloring}. Given an ordered set partition $A := (A_1,
A_2, \ldots, A_m)$ of $V$, define the \Def{rank} of an element $v \in V$ with respect $A$, denoted $\rk[A](v)$, as follows. 
For $v \in V$, let $v_{i_1}, v_{i_2}, \ldots, v_{i_k} = v$ be a longest path in $G$ ending in $v$ such that $v_{i_j} \in A_{i_j}$ for each $j$ and $i_1 < i_2 < \cdots < i_k$, then we define $\rk[A](v) = k$. 
This gives a function $\rk[A]: [d] \to [d]$, which we call the \Def{rank function of $A$}.

We use the same terminology since the rank function of a permutation can be seen as a special case of the rank function of an ordered set partition. 
For a permutation $\pi = [\pi(1) \cdots \pi(d)]$, the ordered set partition consisting of singletons $A:=(\pi(1),\ldots,\pi(d))$, $\rk = \rk[A]$ yields the same rank function; other ordered set partitions can also yield the same rank functions, see Example~\ref{ex:rank}.
The rank with respect to an ordered set partition will be helpful in the proof of
Proposition~\ref{prop:graphicaldecomposition} since for one containment we do not start with a permutation. 

\begin{example}\label{ex:rank}
    Let $G$ be the path graph depicted in Figure~\ref{fig:rankgraphexample}, and let $\pi = [321] \in S_3$. To find $\rk(\pi(3))$, we have positions $1 <3$ so that $\{\pi(1),\pi(3) \} = \{3,1 \}\in E$. 
    However, for the positions $2<3$ we have $\{\pi(2),\pi(3)\} = \{2,1 \} \notin E$. 
    This means $1<3$ is the longest increasing sequence of positions ending at $3$ satisfying our edge condition, showing $\rk(\pi(3))=\rk(1)=2$.
    
    For the same graph, now consider the ordered set partition $A:=(\{3\},\{1,2\})$ of $[3]$. To find $\rk[A](1)$ we need to find a longest path in $G$ using elements in increasing blocks. 
    Here we can pick the path $3-1$, and so $\rk[A](1)=2$.
    In fact, both rank functions above give the same values $\rk(3)=1 =\rk[A](3), \, \rk(2)=2=\rk[A](2), \, \rk(1)=2 = \rk[A](1)$.
\end{example}

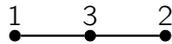
\begin{figure}
    \centering
    \begin{tikzpicture}
    \begin{scope}[xshift=0, yshift=0, scale=1]    
        \vertex[fill](a1) at (0,0) {};
        \vertex[fill](a2) at (1,0) {};
        \vertex[fill](a3) at (2,0) {};

        \node[anchor=south] (b1) at (a1) {$1$};
        \node[anchor=south] (b2) at (a2) {$3$};
        \node[anchor=south] (b2) at (a3) {$2$};
        
        \draw[thick] (a1) --  (a2);
        \draw[thick] (a2) --  (a3);
      
    \end{scope}
\end{tikzpicture}
    \caption{Path graph on three vertices}
    \label{fig:rankgraphexample}
\end{figure}

The rank of a permutation gives rise to the following central $G$-statistic \cite{chunggrahamcoverpol}. 
We say $\pi \in S_d$ has a \Def{$G$-descent} at $i \in [d-1]$ if either
\begin{enumerate}
    \item $\rk(\pi(i)) > \rk(\pi(i+1))$, or 
    \item $\rk(\pi(i)) = \rk(\pi(i+1))$ and $\pi(i) > \pi(i+1)$.
\end{enumerate}
Denote the set of $G$-descents of $\pi$ by $\des_G(\pi) \subseteq [d-1]$. We say $\pi$ has a \Def{$G$-ascent} at $i \in [d-1]$ if either
\begin{enumerate}
    \item $\rk(\pi(i)) < \rk(\pi(i+1))$, or 
    \item $\rk(\pi(i)) = \rk(\pi(i+1))$ and $\pi(i) < \pi(i+1)$.
\end{enumerate}
Denote the set of $G$-ascents of $\pi$ by $\asc_G(\pi) \subseteq [d-1]$. 
Given a permutation $\pi \in S_d$, let the \Def{$G$-major index} be $\maj_G(\pi) := \sum_{i \in \des_G(\pi)} i$. 
Let $\ascnum_G(\pi) := |\asc_G(\pi)|$ and $\desnum_G(\pi):= |\des_G(\pi)|$.
When $E = \varnothing$, these descent statistics stem from the usual descents of a permutation.

\begin{example}
    For the same path graph $G$ as in Example~\ref{ex:rank}, consider again the permutation $\pi = [321] \in S_3$. 
    We computed the ranks $\rk(3)=1, \, \rk(2)=2, \, \rk(1)=2$. 
    Since $\rk(\pi(1)) < \rk(\pi(2))$, we have $1 \in \asc_G(\pi)$. 
    Furthermore, $\rk(\pi(2))=\rk(\pi(3))$ and $\pi(2)=2 > 1 = \pi(3)$, so $2 \in \des_G(\pi)$. 
    Note that the $G$-ascent set in this example is not the same as the usual ascent set on $[321]$ (which is
the empty set).
\end{example}

\subsection{Chromatic generating functions}
We will now describe the polyhedral structure encoding proper colorings of $G$ and then exhibit the connection to $G$-ascents. 

Given an edge $e = ij \in E(G)$, define the hyperplane $H_e := \left\{ x\in\mathbb{R}^d:x_i = x_j\right\}$. 
We can express any proper coloring of $G$ as an integer point $a \in \mathbb{Z}_{ >0 }^d$ such that $a_i
\neq a_j$ whenever $ij \in E$. 
Moreover, we can also see such colorings using at most $n$ colors as lying in the $(n+1)$-dilate of the open
cube $(0,1)^d$ with the graphical hyperplane arrangement $\{H_e\}_{e \in E}$ removed, following the ideas
behind inside-out polytopes~\cite{iop}.
We define
\[
    K_G : = \left\{x \in \mathbb{R}_{> 0}^{d} : \, x_i \neq x_j \text{ for } ij \in E \right\}
\]
and so the integral points of $K_G$ correspond to all proper colorings of $G$.
It will be helpful to record which dilate of the cube $(0,1)^d$ our lattice points are in, as this encodes how many colors are available for a coloring. 
To capture this information, define
\[
    \widehat{K_G} : = \left\{x \in \mathbb{R}^{d+1} : 0< x_1, \ldots, x_d < x_{d+1}, \, x_i \neq x_j \text{ for
} ij \in E \right\},
\]
which we may think of as the homogenization of $K_G$. 
See Figure~\ref{fig:kgexample} for an example.

\begin{figure}
    \centering

\tdplotsetmaincoords{77}{120}

\begin{tikzpicture}
    \begin{scope}[xshift=-75, yshift=0, scale=1]    
        \vertex[fill](a1) at (0,0) {};
        \vertex[fill](a2) at (1,0) {};

        \node[anchor=north] at (a1) {$1$};
        \node[anchor=north] at (a2) {$2$};

        \draw[thick] (a1) -- (a2);
    \end{scope}

    \begin{scope}[xshift=0, yshift=0, scale=0.5]    
            \fill[fill=cyan!30] (0,0) -- (3.75,0) -- (3.75,3.75) -- (0,3.75) -- cycle;
            

            \draw[line width=1.3pt,white] (0,0) -- (4,4);
            \draw[dashed, ->] (0,0) -- (4,4) node[right]{\tiny $x_1=x_2$};

            \draw[line width=1.5pt,white] (0,0) -- (4,0);
            \draw[line width=1.5pt,white] (0,0) -- (0,4);

            \draw[->,thick] (0,0) -- (4,0) node[right]{\small $x_1$};
            \draw[->,thick] (0,0) -- (0,4) node[left]{\small $x_2$};

            \node[circle, draw, inner sep=0pt, minimum size=2pt, fill] (a1) at (1,2) {};
            \node[circle, draw, inner sep=0pt, minimum size=2pt, fill] (a2) at (3,1) {};
            \node[circle, draw, inner sep=0pt, minimum size=2pt, fill] (a3) at (3,2) {};

            \node[circle, draw, inner sep=0pt, minimum size=2pt, fill] (a4) at (2,1) {};
            \node[circle, draw, inner sep=0pt, minimum size=2pt, fill] (a5) at (1,3) {};
            \node[circle, draw, inner sep=0pt, minimum size=2pt, fill] (a6) at (2,3) {};

            \node[circle, draw, inner sep=0pt, minimum size=2pt,fill=white] (a7) at (1,1) {};
            \node[circle, draw, inner sep=0pt, minimum size=2pt,fill=white] (a8) at (2,2) {};
            \node[circle, draw, inner sep=0pt, minimum size=2pt,fill=white] (a9) at (3,3) {};

            \node[circle, draw, inner sep=0pt, minimum size=2pt,fill=white] (b1) at (0,0) {};
            \node[circle, draw, inner sep=0pt, minimum size=2pt,fill=white] (b2) at (1,0) {};
            \node[circle, draw, inner sep=0pt, minimum size=2pt,fill=white] (b3) at (2,0) {};
            \node[circle, draw, inner sep=0pt, minimum size=2pt,fill=white] (b4) at (3,0) {};
            \node[circle, draw, inner sep=0pt, minimum size=2pt,fill=white] (b5) at (0,1) {};
            \node[circle, draw, inner sep=0pt, minimum size=2pt,fill=white] (b6) at (0,2) {};
            \node[circle, draw, inner sep=0pt, minimum size=2pt,fill=white] (b7) at (0,3) {};

            \node[cyan] at (5,2) {\small $K_G$};
    \end{scope}

    \begin{scope}[xshift=150, yshift=0,scale=0.85,tdplot_main_coords,
	facet/.style={fill=orange,fill opacity=0.3}]
            \coordinate (A) at (0,0,1);
            \coordinate (B) at (1,0,1);
            \coordinate (C) at (0,1,1);
            \coordinate (D) at (1,1,1);

            \coordinate (1) at (0,0,2);
            \coordinate (2) at (2,0,2);
            \coordinate (3) at (0,2,2);
            \coordinate (4) at (2,2,2);

            \coordinate (e) at (0,0,3);
            \coordinate (f) at (3,0,3);
            \coordinate (g) at (0,3,3);
            \coordinate (h) at (3,3,3);

            \coordinate (e2) at (0,0,3.2);
            \coordinate (f2) at (3.2,0,3.2);
            \coordinate (g2) at (0,3.2,3.2);
            \coordinate (h2) at (3.2,3.2,3.2);

            \node[right, cyan] at (g2) { \small $\widehat{K_G}$};

            \draw[->,thick] (-0.25,0,0) -- (3.5,0,0) node[left]{\tiny $x_1$}; 
            \draw[->,thick] (0,-0.25,0) -- (0,3.25,0) node[right]{\tiny $x_2$};
            \draw[->,thick] (e) -- (0,0,3.5) node[above]{\tiny $x_3$};

            \fill[facet, cyan] (0,0,0) --(e) -- (f) -- cycle;

            \fill[facet, cyan] (0,0,0) --(e) -- (g) -- cycle;

            \fill[BrickRed,opacity=0.3] (0,0,0) -- (h) -- (-0.25,-0.25,3) -- (-0.25,-0.25,0) -- cycle;
            \draw[BrickRed,opacity=0.25,line width =1pt] (h) -- (-0.25,-0.25,3) -- (-0.25,-0.25,0) -- (0,0,0);

            \draw[thick, color=cyan, dashed] (A) -- (B) -- (D) -- (C) -- cycle;

            \draw[thick, color=cyan,dashed] (1) -- (2) -- (4) -- (3) -- cycle;

            \draw[->,cyan,thick,dashed] (0,0,0) -- (e2) node[left]{}; 
            \draw[->,cyan,thick,dashed] (0,0,0) -- (f2) node[right]{};
            \draw[->,cyan,thick,dashed] (0,0,0) -- (g2) node[above]{};
            \draw[->,cyan,thick,dashed] (0,0,0) -- (h2) node[above]{};

            \node[circle, draw, inner sep=0pt, minimum size=2pt] at (1,1,2) {};
            \node[circle, draw, inner sep=0pt, minimum size=2pt] at (0,0,2) {};
            \node[circle, draw, inner sep=0pt, minimum size=2pt] at (0,0,1) {};

            \node[circle, draw, inner sep=0pt, minimum size=2pt] at (0,1,2) {};
            \node[circle, draw, inner sep=0pt, minimum size=2pt] at (1,0,2) {};

            \fill[facet, cyan] (0,0,0) --(g) -- (h) -- cycle;
            \fill[facet, cyan] (0,0,0) --(h) -- (f) -- cycle;

            \node[circle, draw, inner sep=0pt, minimum size=2pt] at (1,1,3) {};
            \node[circle, draw, inner sep=0pt, minimum size=2pt, fill] at (2,1,3) {};
            \node[circle, draw, inner sep=0pt, minimum size=2pt, fill] at (1,2,3) {};
            \node[circle, draw, inner sep=0pt, minimum size=2pt] at (2,2,3) {};
            
            \node[circle, draw, inner sep=0pt, minimum size=2pt] at (e) {};
            \node[circle, draw, inner sep=0pt, minimum size=2pt] at (f) {};
            \node[circle, draw, inner sep=0pt, minimum size=2pt] at (g) {};
            \node[circle, draw, inner sep=0pt, minimum size=2pt] at (h) {};
            \node[circle, draw, inner sep=0pt, minimum size=2pt] at (1,0,3) {};
            \node[circle, draw, inner sep=0pt, minimum size=2pt] at (2,0,3) {};
            \node[circle, draw, inner sep=0pt, minimum size=2pt] at (3,1,3) {};
            \node[circle, draw, inner sep=0pt, minimum size=2pt] at (3,2,3) {};
            \node[circle, draw, inner sep=0pt, minimum size=2pt] at (2,3,3) {};
            \node[circle, draw, inner sep=0pt, minimum size=2pt] at (3,1,3) {};
            \node[circle, draw, inner sep=0pt, minimum size=2pt] at (0,2,3) {};
            \node[circle, draw, inner sep=0pt, minimum size=2pt] at (0,1,3) {};

            \node[circle, draw, inner sep=0pt, minimum size=2pt] at (B) {};
            \node[circle, draw, inner sep=0pt, minimum size=2pt] at (C) {};
            \node[circle, draw, inner sep=0pt, minimum size=2pt] at (D) {};
            \node[circle, draw, inner sep=0pt, minimum size=2pt] at (2) {};
            \node[circle, draw, inner sep=0pt, minimum size=2pt] at (3) {};
            \node[circle, draw, inner sep=0pt, minimum size=2pt] at (4) {};

            \node[circle, draw, inner sep=0pt, minimum size=2pt] at (2,1,2) {};
            \node[circle, draw, inner sep=0pt, minimum size=2pt] at (1,2,2) {};

            \fill[BrickRed,opacity=0.3] (0,0,0) -- (h) -- (3.25,3.25,3) -- (3.25,3.25,0) -- cycle;
            \draw[BrickRed,opacity=0.25,line width =1pt] (h) -- (3.25,3.25,3) -- (3.25,3.25,0) -- (0,0,0);

            \node[circle, draw, inner sep=0pt, minimum size=2pt] at (0,0,0) {};

            \node[right] at (3.25,3.25,0) {\tiny $x_1=x_2$};
    \end{scope}
\end{tikzpicture}
    \caption{For $G$ the path graph on two vertices, the center figure shows $K_G$, and the right figure shows $\widehat{K_G}$.}\label{fig:kgexample}
\end{figure}
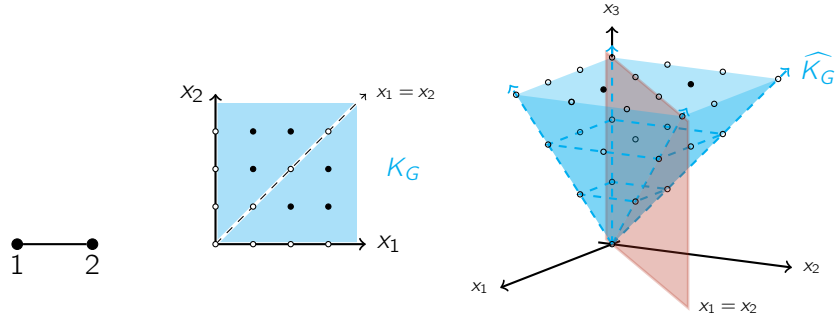

For these polyhedral structures, we now create a disjoint polyhedral decomposition using $G$-ascents. 
We will turn our focus to $\widehat{K_G}$ for the rest of this section, as this will connect to the $q$-chromatic polynomial, but the arguments and proofs are analogous for $K_G$, which gives the $G$-partition generating function defined by Bajo et al.~\cite{qchromatic}. 
For $\pi \in S_d$ define the half-open simplicial cone
\[
    \widehat{\Delta_\pi^G} := \left\{ x \in \mathbb{R}^{d+1} : 
    \begin{array}{l}
        0 < x_{\pi(1)} \leq x_{\pi(2)} \leq \cdots \leq x_{\pi(d)} < x_{d+1} \\ x_{\pi(i)} < x_{\pi(i+1)} \text{ whenever } i \in \asc_G(\pi)  
    \end{array} \right\}.
\]
We denote $\widehat{\Delta_\pi^G}$ by $\widehat{\Delta_\pi}$ when $G$ is clear from the context. 

\begin{proposition}\label{prop:graphicaldecomposition}
    Let $G = ([d],E)$ be a graph. Then
    \[
        \widehat{K_G} = \biguplus_{\pi \in S_d} \widehat{\Delta_{\pi}^G} \, .
    \]
\end{proposition}

\begin{proof}
We remark that the general idea of the following argument is essentially given by Steingrímsson \cite[Theorem~6]{steingrimssoncoloring}.
We first show $\widehat{K_G} \supseteq \biguplus_{\pi \in S_d} \widehat{\Delta_{\pi}^G}$.
Let $\pi \in S_d$; it suffices prove $\widehat{\Delta_{\pi}} \cap H_e = \emptyset$ for all $e \in E$. 
Let $e=\{v_1,v_2\} \in E$, so there exist $i_1,i_2 \in [d]$ such that $\pi(i_1) = v_1$ and $\pi(i_2) = v_2$; without loss of generality suppose $i_1 < i_2$. 
This means $\rk(v_2) \geq \rk(v_1)+1$ since when computing the rank for $\pi(i_2)$, a longest path in the definition of rank of $\pi(i_2)$ must be at least as long as the longest path defining the rank of $\pi(i_1)$ with the edge $v_1 v_2 \in E$. 
Towards a contradiction, if all positions $j$ where $i_1 \leq j < i_2$ are $G$-descents, then ranks must be
weakly decreasing, $\rk(v_1) \geq \rk(\pi(i_1 +1)) \geq \cdots \geq \rk(v_2)$, which gives a contradiction. 
So there must exist a $G$-ascent at some position $j$ where $i_1 \leq j < i_2$, yielding  
the strict inequality $x_{v_1} < x_{v_2}$ in the definition of $\widehat{\Delta_{\pi}}$. As the edge was arbitrary, $\widehat{\Delta_{\pi}} \cap H_e = \emptyset$ for all $e \in E$.

We now show that  $\widehat{K_G} \subseteq \biguplus_{\pi \in S_d} \widehat{\Delta_{\pi}^G}$, that is, a given $a := (a_1,\ldots,a_{d+1}) \in \widehat{K_G}$ satisfies $a \in \widehat{\Delta_{\pi}}$ for some unique $\pi \in S_d$. 
Let $b_1 < b_2 < \cdots < b_k<b_{k+1}$ be all the distinct values among the entries of $a$.
By definition of $\widehat{K_G}$, we know $a_{d+1}$ is the only term achieving the highest value $b_{k+1}$.
For each $j \in [k]$ define $A_{j} :=\{i  \in [d] : a_i = b_j\}$. 
This gives an ordered set partition $(A_j,\ldots,A_{k})$ of the vertex set $[d]$.
To form our permutation, we first place the blocks in order $A_1,\ldots,A_k$. 
Then, within each block we order the vertices decreasingly by their rank with respect to the ordered set partition $(A_1,\ldots,A_k)$.
Lastly, we order consecutive vertices that share the same rank decreasingly by their label. 
Since each block of $A$ are vertices which share the same value from the point $a$, each block is an independent set in $G$.
Thus, ordering decreasingly by rank in each block is well-defined because the rank function depends only on the previous blocks, and the lack of edges within each block implies that reordering vertices within a block does not change the rank function.
See Example~\ref{ex:uniquepermcone} for an example of this process. 
Ordering the vertices in this manner gives a well-defined permutation $\pi \in S_d$.

We next claim $\rk[A] = \rk[\pi]$. 
Let $\pi(i) \in [d]$ and set $r := \rk(\pi(i))$ and $p :=\rk[A](\pi(i))$. 
By definition of $r$, there exist positions $i_1 < i_2 < \dots < i_r = i$ such that $\{\pi(i_j), \pi(i_{j+1}) \} \in E$ for $1 \le j < r$.
So the path formed by $\pi(i_1) \pi(i_2) \dots \pi(i)$ must use vertices in different blocks of $A$ of increasing index.
Since $p$ is defined from the longest such path satisfying this condition, we have $r \leq p$. 
Similarly, by the definition of $p$, there exists a longest path $v_{j_1}, v_{j_2}, \ldots, v_{j_p} = \pi(i)$ ending in $\pi(i)$ such that $v_{j_k} \in A_{j_k}$ for each $k$ and $j_1 < j_2 < \cdots < j_p$.
These vertices must appear in increasing positions of $\pi$ by construction since they are in blocks of $A$ of increasing index. 
Since $r$ is defined from the longest such path satisfying this condition, we have $p \leq r$.
This shows $r = p$.
Since this was for an arbitrary element $\pi(i)$, the rank functions are equal.

Observe next that by construction of $\pi$ and since the rank functions coincide, $G$-ascents are only possible between blocks in our ordering.
Note that it might be the case that only some of these positions are $G$-ascents, but it is certain that no $G$-ascents occur within a block $A_j$ of elements in the ordering.
Since $a$ must satisfy equality between elements in the order within a fixed block $A_j$, and $a$ is able to
satisfy strict inequalities between blocks in our ordering, where $a_{d+1}$ is strictly greater than any other entry, we have $a \in \widehat{\Delta_\pi}$.

Finally, we show that $a$ cannot be an element of any $\widehat{\Delta_\tau}$ for $\tau\neq \pi$.
Specifically, since the elements of $a$ within a block $A_j$ are all equal, those elements must occur in $G$-descending order. 
Since the values of $a$ are distinct among blocks, for any $\tau$ with $a\in \widehat{\Delta_\tau}$, it must be that the blocks of elements $A_j$ appear in strictly increasing order with respect to $\tau$. 
Thus, we are forced to have the blocks in order, as well as to avoid any $G$-ascents within a block, which leads that any such $\tau$ must equal $\pi$, and our proof is complete.
\end{proof}

\begin{example}\label{ex:uniquepermcone}
    We will give an example of how a point in $\widehat{K_G}$ lies in some $\widehat{\Delta_\pi^G}$ for some
$\pi \in S_d$, as suggested by Proposition~\ref{prop:graphicaldecomposition}.
    Let $G$ be the path graph from Example~\ref{ex:rank}. We have $a= (a_1,a_2,a_3,a_4) := (1.1,1.1,2,3)  \in \widehat{K_G}$. 
    To construct a permutation $\pi$ such that $a \in \widehat{\Delta_\pi}$, we have distinct values $a_1=1.1 <2=a_3<3=a_4$.
    This gives an ordered set partition $A=(\{1,2 \},\{3\})$ of $[3]$.
    To construct our permutation, we place our blocks in order of the ordered set partition and then order within each block by rank with respect to $A$ or labels if ranks are equal.
    For instance, in the first block note that $1$ and $2$ both have rank $1$ with respect to $A$, and since $1<2$ then we have $21$ first appearing. 
    The block $\{3\}$ comes after and so we have the permtuation $[213] \in S_3$.
    We have $\asc_G([213]) = \{2\}$, and so $a \in \widehat{\Delta_{[213]}} = \{(x_1,x_2,x_3,x_4) \in \mathbb{R}^4 : 0<x_2 \leq x_1 < x_3 < x_4\}$. 

    We give another example of this process. Let $H$ be the bowtie graph given in Figure~\ref{fig:gsequence}.
    Now consider $b=(b_1,b_2,b_3,b_4,b_5,b_6) := (2,3,2,1,4,5) \in \widehat{K_H}$.
    To construct a permutation $\tau$ such that $b \in \widehat{\Delta_\tau}$ we first create the ordered set partition of $[5]$ from our distinct values.
    This ordered set partition is $B=(\{4\},\{1,3\},\{2\},\{5\})$. 
    To construct $\tau$ we place blocks in order and in order to decide the order within block $\{1,3\}$ we must decrease by rank with respect to this ordered set partition. We have $\rk[B](1) = 2 > 1=\rk[B](3)$ since $14 \in E(H)$. 
    This means $\tau = [41325]$.
    Since $\asc_H(\tau) = \{1,3,4\}$, then $b \in \widehat{\Delta_{[41325]}} = \{(x_1,x_2,x_3,x_4) \in \mathbb{R}^4 : 0<x_4 < x_1 \leq x_3 < x_2 < x_5 < x_6\}$.
\end{example}

We now study the main multivariate generating function behind proper colorings. 
The \Def{integer point transform} of a set $S \subseteq \mathbb{R}^n$ is 
\[
    \sigma_S(z_1,\ldots,z_{n}) := \sum_{m \in S \cap \mathbb{Z}^n} z^m
\]
where $z^m = z_1^{m_1}z_2^{m_2} \cdots z_n^{m_n}$ for $m=(m_1,\ldots,m_n)$.
We can compute $\sigma_{\widehat{K_G}}$ from $\sigma_{\widehat{\Delta_\pi}}$ via our decomposition in
Proposition~\ref{prop:graphicaldecomposition}, and each $\sigma_{\widehat{\Delta_\pi}}$ can be computed by
standard methods (see, e.g.,~\cite{crt}).

\begin{theorem}\label{thm:multivariablechromaticgenerating}
    Given a graph $G=([d],E)$, 
    \[
        \sigma_{\widehat{K_G}}(z_1,\ldots,z_d,z_{d+1}) = \sum_{\pi \in S_d} \frac{ z_{d+1}^2 z_1 \cdots z_d \prod_{j \in \asc_G(\pi)} z_{\pi(j+1)} \cdots z_{\pi(d)} z_{d+1} }{(1-z_{d+1}) \prod_{i=1}^d (1-z_{\pi(i)} \cdots z_{\pi(d)} z_{d+1})}.
    \]
\end{theorem}
\begin{proof}
    Given $\pi \in S_d$, let $v^\pi_i$ be the ray of $\widehat{\Delta_\pi}$ given by the vector $v^\pi_i=x\in \mathbb{R}^{d+1}$ with entries 
    \[
    x_{\pi(i+1)} =\dots = x_{\pi(d)} = x_{d+1} = 1
    \]
    and remaining entries equal to $0$, where $v^\pi_d$ sets only $x_{d+1}$ to $1$. 
    So, 
        \[
            \widehat{\Delta_\pi} = \mathbb{R}_{>0} v^\pi_0 + \mathbb{R}_{>0} v^\pi_d + \sum_{j \in
            \asc_G(\pi)}\mathbb{R}_{>0} v^\pi_j + \sum_{j \in \des_G(\pi)} \mathbb{R}_{\geq 0} v^\pi_j \, .
        \]
    This half-open simplicial cone comes with the fundamental parallelepiped 
    \[
        \square_{\widehat{\Delta_\pi}} = (0,1] v^\pi_0 + (0,1]v^\pi_d + \sum_{j \in \asc_G(\pi)} (0,1]
        v^\pi_j + \sum_{j \in \des_G(\pi)} [0,1) v^\pi_j  \, .
    \]
    By Proposition~\ref{prop:graphicaldecomposition},
    \[
        \sigma_{\widehat{K_G}}(z) = \sum_{\pi \in S_d} \sigma_{\widehat{\Delta_\pi}}(z) = \sum_{\pi \in S_d}
\frac{\sigma_{\square_{\widehat{\Delta_\pi}}}(z)}{\prod_{i=0}^d (1-z^{v^\pi_i}) } = \sum_{\pi \in S_d}
\frac{z^{v^\pi_0} z^{v^\pi_d} \prod_{j \in \asc_G(\pi)} z^{v^\pi_j}} {\prod_{i=0}^d (1-z^{v^\pi_i}) } \, .
    \]
    Since $v^\pi_0= (1,\ldots,1)$ and $v^\pi_d = (0,\ldots,0,1)$ in $\mathbb{R}^{d+1}$ for any $\pi \in S_d$,
    \[
        \sigma_{\widehat{K_G}}(z) = \sum_{\pi \in S_d} \frac{(z_1 \cdots z_{d+1} ) (z_{d+1}) \prod_{j \in \asc_G(\pi)} z_{\pi(j+1)} \cdots z_{\pi(d)} z_{d+1}} {(1-z_{d+1}) \prod_{i=1}^d 1-z_{\pi(i)} \cdots z_{\pi(d)} z_{d+1}}
    \]
    as desired.
\end{proof}

For any linear form $\lambda \in \mathbb{Z}^d$, we now specialize this integer-point transform to find the generating function for the $q$-chromatic polynomial.

\begin{theorem}\label{thm:genfuncweighted}
    Given a graph $G = ([d],E)$ and a linear form $\lambda=(\lambda_1,\ldots,\lambda_d) \in \mathbb{Z}^d$, let
 $\Lambda := \sum_{i \in [d]} \lambda_i$. Then
\begin{align*}
        \sum_{n \geq 0} \chi^\lambda_G(q,n) z^n
        &= z^{-1} \, \sigma_{\widehat{K_G}}(q^{\lambda_1},\ldots,q^{\lambda_d},z) \\
        &= \sum_{\pi \in S_d} \frac{q^{\Lambda + \sum_{j \in \asc_G(\pi)} \lambda_{\pi(j+1)} + \cdots + \lambda_{\pi(d)}} z^{\ascnum_G(\pi)+1} }{(1-z)(1-q^{\lambda_{\pi(d)}}z) (1-q^{\lambda_{\pi(d)} + \lambda_{\pi(d-1)} } z) \cdots (1-q^\Lambda z)}  \, .
\end{align*} 

\end{theorem}

\begin{proof}
    We set $z_i = q^{\lambda_i}$ for $i \in [d]$ and $z_{d+1} = z$ and consider what this does for the integer point transform by definition.
    Recall that each lattice point in $K_G$ can be seen as a lattice point in the $(n+1)$-dilate of the cube
$(0,1)^d$, removing points on $H_e$ for all $e \in E$, and any such a point corresponds to a proper $n$-coloring. With this perspective,
\begin{align*}
        \sigma_{\widehat{K_G}} \left(q^{\lambda_1},\ldots,q^{\lambda_d},z \right)
        &= \sum_{m \in \widehat{K_G} \cap \mathbb{Z}^{d+1}} q^{\lambda_1 m_1 + \cdots +\lambda_d m_d} z^{m_{d+1}} 
         =\sum_{n \geq 0} \sum_{\substack{\text{proper } \\n-\text{colorings } c}} q^{ \lambda_1 c(1)+ \cdots +
\lambda_d c(d)} z^{n+1} \\
        &= z \cdot \sum_{n \geq 0}\chi_G^\lambda(q,n) z^n.
\end{align*}    
    Next, we specialize the expression given by Theorem~\ref{thm:multivariablechromaticgenerating}:
\begin{align*}
        \sigma_{\widehat{K_G}}(q^{\lambda_1},\ldots,q^{\lambda_d},z)
        &= \sum_{\pi \in S_d} \frac{ z^2 q^{\lambda_1} \cdots q^{\lambda_d} \prod_{j \in \asc_G(\pi)}
q^{\lambda_{\pi(j+1)}} \cdots q^{\lambda_{\pi(d)}} z }{(1-z) \prod_{i=1}^d (1-q^{\lambda_{\pi(i)}} \cdots
q^{\lambda_{\pi(d)}} z)} \\
         &= z \sum_{\pi \in S_d} \frac{ q^{ \Lambda + \sum_{j \in \asc_G(\pi)} \lambda_{\pi(j+1)} + \cdots +
\lambda_{\pi(d)}} z^{\ascnum_G(\pi)+1} }{(1-z) \prod_{i=1}^d (1-q^{\lambda_{\pi(i)}} \cdots q^{\lambda_{\pi(d)}}
z)} \, .
\end{align*} 
    Comparing the two expressions for $\sigma_{\widehat{K_G}}(q^{\lambda_1},\ldots,q^{\lambda_d},z)$ yields
    \[
        \sum_{n \geq 0}\chi_G^\lambda(q,n) z^n =\sum_{\pi \in S_d} \frac{ q^{ \Lambda + \sum_{j \in \asc_G(\pi)}
\lambda_{\pi(j+1)} + \cdots + \lambda_{\pi(d)}} z^{\ascnum_G(\pi)+1} }{(1-z) \prod_{i=1}^d
(1-q^{\lambda_{\pi(i)}} \cdots q^{\lambda_{\pi(d)}} z)} \, . \qedhere 
    \]
\end{proof}

\begin{remark}
    Further setting $q=1$ gives a generating-function formula for the chromatic polynomial due to Chung--Graham~\cite{chunggrahamcoverpol};
see also Corollary~\ref{cor:chunggraham} below.
\end{remark}

\subsection{Alternative characterizations}
We will now reformulate some expressions coming from the generating function of the $q$-chromatic polynomial
using $G$-sequences, which were introduced by Steingrímsson~\cite{steingrimssoncoloring}. They give an alternate
perspective to our chromatic generating functions, as well as insight for future results.

Let $G=([d],E)$ be a graph and $\pi = [\pi(1) \cdots \pi(d)] \in S_d$. 
The \Def{$G$-sequence} of $\pi$ is an ordered set partition of $[d]$ defined as follows:
\begin{itemize}
    \item If $\asc_G(\pi) = \emptyset$, then our ordered set partition is the singular block $A_1 := [d]$.
    \item If $\asc_G(\pi) =\{i_1 < \cdots < i_{\ascnum_G(\pi)} \}$, then our ordered set partition is $(A_1, \ldots, A_{\ascnum_G(\pi)+1})$
    where
\begin{align*}
        A_1 &:= \{\pi(1), \pi(2),\ldots,\pi(i_1)\} \\
        A_m &:= \{\pi(i_{m-1} +1), \ldots, \pi(i_{m}) \} \ \text{ for } \ 2 \le m \le \ascnum_G(\pi) \\
        A_{\ascnum_G(\pi)+1} &:= \{\pi(i_{\ascnum_G(\pi)} +1),\ldots, \pi(d)\} \, .
\end{align*}
\end{itemize}

\begin{figure}
    \centering
    \begin{tikzpicture}
    \begin{scope}[xshift=0, yshift=0, scale=1]    
        \vertex[fill](a1) at (0,0) {};
        \vertex[fill](a2) at (1,0.5) {};
        \vertex[fill](a3) at (1,-0.5) {};
        \vertex[fill](a4) at (-1,0.5) {};
        \vertex[fill](a5) at (-1,-0.5) {};

        \node[anchor=south] (b1) at (a1) {$2$};
        \node[anchor=south] (b2) at (a2) {$1$};
        \node[anchor=north] (b3) at (a3) {$4$};
        \node[anchor=south] (b4) at (a4) {$5$};
        \node[anchor=north] (b5) at (a5) {$3$};

        \draw[thick] (a1) --  (a2) -- (a3) -- (a1) -- cycle;
        \draw[thick] (a1) --  (a4) -- (a5) -- (a1) -- cycle;
      
    \end{scope}

    \begin{scope}[xshift=100, yshift=0, scale=1]    
        \vertex[fill,BrickRed](a1) at (0,0) {};
        \vertex[fill,cyan](a2) at (1,0.5) {};
        \vertex[fill,violet](a3) at (1,-0.5) {};
        \vertex[fill,orange](a4) at (-1,0.5) {};
        \vertex[fill,cyan](a5) at (-1,-0.5) {};

        \node[anchor=south,BrickRed] (b1) at (a1) {$2$};
        \node[anchor=south,cyan] (b2) at (a2) {$1$};
        \node[anchor=north,violet] (b3) at (a3) {$4$};
        \node[anchor=south,orange] (b4) at (a4) {$3$};
        \node[anchor=north,cyan] (b5) at (a5) {$1$};

        \draw[thick] (a1) --  (a2) -- (a3) -- (a1) -- cycle;
        \draw[thick] (a1) --  (a4) -- (a5) -- (a1) -- cycle;
      
    \end{scope}
\end{tikzpicture}
    \caption{The bowtie graph and its $G$-sequence coloring for $\pi=[31254]$.}
    \label{fig:gsequence}
\end{figure}
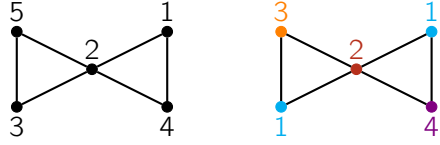

\begin{example}\label{ex:gsequence}
    Consider the bowtie graph $G$ in Figure~\ref{fig:gsequence} and the permutation $\pi = [31254] \in S_5$. Then  $\rk(3) = 1$, $\rk(1)=1$, $\rk(2)=2$, $\rk(5)=3$, and $\rk(4)=4$. With these rank values we can find that $\asc_G(\pi) = \{ 2,3,4\}$, and so our $G$-sequence is the ordered set partition $31/2/5/4$.
\end{example}

Moreover, Steingrímsson proved the following \cite[Lemma~4]{steingrimssoncoloring}.

\begin{lemma}[Steingrímsson]\label{lem:independentgsequence}
    Let $(A_1,\ldots, A_{\ascnum_G(\pi)+1})$ be the $G$-sequence of a permutation $\pi \in S_d$. 
    Each $A_j$ is an independent set in $G$, i.e., no two vertices in $A_j$ are adjacent, for any~$j$.
\end{lemma}

It is straightforward to see why Lemma~\ref{lem:independentgsequence} holds, as an edge within a block would force the existence of a new ascent.
Due to Lemma~\ref{lem:independentgsequence}, we can use the $G$-sequence to define a proper coloring, which we will show relates to the numerators in the expression of the $q$-chromatic generating function.
For a graph $G$ on $[d]$ and $\pi \in S_d$, suppose $\pi$ has $G$-sequence $(A_1,\ldots,A_{\ascnum_G(\pi)+1})$.
We define the \Def{$G$-sequence coloring of $\pi$} as 
\[
    w_{\pi}: [d] \to [\ascnum_G(\pi)+1] \quad \text{where } w_\pi(i) :=j \text{ if } i \in A_j.
\]
The $G$-sequence coloring of Example~\ref{ex:gsequence} is shown in Figure~\ref{fig:gsequence}.

We now prove that the numerators found in Theorem~\ref{thm:genfuncweighted} can be computed via~$w_{\pi}$.

\begin{proposition}\label{prop:weightedcolorsum}
    Let $G$ be a graph on $[d]$, let $\pi \in S_d$ be a permutation and  $\lambda \in \mathbb{Z}_{>0}^d$ a linear
form. Then
    \[
        \Lambda + \sum_{j \in \asc_G(\pi)} \lambda_{\pi(j+1)} + \cdots + \lambda_{\pi(d)} = \sum_{i \in [d]}
\lambda_i w_{\pi}(i) \, .
    \]
\end{proposition}
\begin{proof}
    For $\pi = [\pi(1) \pi(2) \cdots \pi(d)] \in S_d$, let $i_2 < \cdots < i_{\ascnum_G(\pi)+1}$ be the elements
of $\asc_G(\pi)$, and set $i_1 := 0$, $i_{\ascnum_G(\pi)+2}:=d$. Then by definition $A_j=\{\pi({i_j +1}),
\pi({i_j +2}), \ldots, \pi({i_{j+1}})\}$, and so position $i_2$ is given by $|A_1|$, $i_3$ is given by $|A_1| +
|A_2|$, and so on. Thus 
    \begin{alignat*}{2}
        \Lambda + \sum_{j \in A} \lambda_{\pi(j+1)} + \cdots + \lambda_{\pi(d)} &= \Lambda + \lambda_{\pi(i_2 +1)} + \cdots + \lambda_{\pi(i_3)}&+ \lambda_{\pi(i_3+1)}+\cdots + \lambda_{\pi(i_{\ascnum_G(\pi) +1 } +1 )} + \cdots + \lambda_{\pi(d)} \\
        & &+\lambda_{\pi(i_3+1)}+\cdots + \lambda_{\pi(i_{\ascnum_G(\pi)+1 } +1 )} + \cdots + \lambda_{\pi(d)}\\
        & &\vdots \hspace{4.5em}\\
        & &+ \lambda_{\pi(i_{\ascnum_G(\pi)+1 } +1 )} + \cdots + \lambda_{\pi(d)} \\
        &= \Lambda + \sum_{k = 1}^{\ascnum_G(\pi)+1} (k-1) (\lambda_{\pi(i_k +1)} + \cdots + \lambda_{\pi(i_{k+1})}) \hspace{-10em} \\
        &= \Lambda + \sum_{k = 1}^{\ascnum_G(\pi)+1} (k-1) \left( \sum_{a \in A_k} \lambda_a \right)\hspace{-10em} \\
        &= \Lambda - \sum_{k = 1}^{\ascnum_G(\pi)+1} \sum_{a \in A_k} \lambda_a +  \sum_{k = 1}^{\ascnum_G(\pi)+1}  \sum_{a \in A_k} k \lambda_a \hspace{-10em} \\
        &= \Lambda - \Lambda + \sum_{k = 1}^{\ascnum_G(\pi)+1}  \sum_{a \in A_k} w_\pi(a) \lambda_a \hspace{-10em} \\
        &=  \sum_{i \in [d]} w_{\pi}(\pi(i)) \lambda_{\pi(i)} \, .
    \end{alignat*}
    The third equality holds as $A_j=\{\pi({i_j +1}), \pi({i_j +2}), \ldots, \pi({i_{j+1}})\}$. 
    The fifth equality holds since $A_1,A_2,\ldots$ partition $[d]$; hence summing over all sets in this
partition gives $\Lambda$. 
    Moreover, each element $a \in A_k$ is given the weight $k = w_{\pi,A}(a)$ and $\lambda_{a}$. 
    We can further simplify \[ \sum_{i \in [d]} w_{\pi}(\pi(i)) \lambda_{\pi(i)} = \sum_{\pi^{-1}(j) \in [d]} w_{\pi}(j) \lambda_{j} = \sum_{i \in [d]} w_\pi(i)
\lambda_i \, , \] proving the desired equation. 
\end{proof}

\begin{corollary}
    For $G$ a graph on $[d]$ and a linear form $\lambda \in \mathbb{Z}^d$,
    \[
        \sum_{n \geq 0} \chi^\lambda_G(q,n) z^n = \sum_{\pi \in S_d} \frac{q^{\sum_{i \in [d]} \lambda_i
w_\pi(i)} z^{\ascnum_G(\pi)+1} }{(1-z)(1-q^{\lambda_{\pi(d)}}z) (1-q^{\lambda_{\pi(d)} + \lambda_{\pi(d-1)} } z)
\cdots (1-q^\Lambda z)}  \, .
    \]
\end{corollary}


\section{The linear form $\lambda=\mathbbm{1}$}\label{sec:linearform1}

In this section, we focus our attention on the special case $\lambda=\mathbbm{1}$.
We investigate consequences of our multivariate generating function identities and connections with quasisymmetric functions.

\subsection{A combinatorial interprentation of $q$-binomial basis representations.}
We will now apply the results from Section~\ref{sec:gstats} to the special linear form $\lambda = \mathbbm{1}
\in \mathbb{Z}^d$, giving an expression of the coefficients of $\chi_G^\lambda(q,n)$ in the basis $\left\{\qbinom{n+k}{d} : 0 \leq k \leq d\right\}$. 
In addition, via $G$-sequences we can omit some summands in this basis expression.

\begin{theorem}\label{thm:qchromaticallones}
    For a graph $G = ([d],E)$, 
    \[
        \sum_{n \geq 0} \chi^\mathbbm{1}_G(q,n) z^n =  \frac{\sum_{\pi \in S_d} q^{d + \sum_{j \in \asc_G(\pi)}
d-j} z^{\ascnum_G(\pi)+1} }{(1-z)(1-q z) (1-q^{2} z) \cdots (1-q^d z)} \, .
    \]
\end{theorem}

\begin{proof}
    Upon setting $\lambda=\mathbbm{1}$ in Theorem~\ref{thm:genfuncweighted}, all permutations have the same denominator. Moreover, for each permutation $\pi \in S_d$, 
    \[
        q^{\Lambda + \sum_{j \in \asc_G(\pi)} \lambda_{\pi(j+1)} + \cdots + \lambda_{\pi(d)}}
z^{\ascnum_G(\pi)+1}  = q^{d + \sum_{j \in \asc_G(\pi)} d-j } z^{\ascnum_G(\pi)+1}. \qedhere 
    \]
\end{proof}

\begin{corollary}\label{cor:qchromaticqbinomial}
    Let $\xi$ be the chromatic number of~$G$. Then
    \[
        \chi_G^\mathbbm{1}(q,n) = \sum_{j=0}^{d-\xi} \left( \qbinom{n+j}{d} \sum_{\substack{\pi \in S_d \\
\desnum_G(\pi) = j}} q^{d+\sum_{j \in \asc_G(\pi)} d-j} \right) .
    \]
\end{corollary} 

\begin{proof}
    For $\pi \in S_d$, set $\alpha_\pi = d+\sum_{j \in \asc_G(\pi)} d-j$, and let 
    \[
        \sum_{\pi \in S_d} q^{\alpha_\pi} z^{\ascnum_G(\pi)+1} = \sum_{i=0}^m a_i(q)z^i \in
\mathbb{Z}[q,z] \, ,
    \]
    so $a_i(q) = \sum q^{\alpha_\pi}$ where we sum over all permutations $\pi$ such that $\ascnum_G(\pi)+1 = i$,
i.e., $\ascnum_G(\pi) = i-1$. Hence
    \begin{align*}
        \sum_{t \geq 0} \chi_G^\mathbbm{1}(q,t) z^t &= \frac{a_0(q) + \cdots + a_m(q) z^m}{(1-z)(1-qz) \cdots (1-q^dz)} \\
        &= \left(a_0(q) + \cdots + a_m(q) z^m \right) \left( \sum_{i \geq 0} \qbinom{d+i}{d} z^i\right) \\
        &= \sum_{j \geq 0}\left( \sum_{i=0}^j \qbinom{d+j-i}{d} a_i(q)\right) z^j,
    \end{align*}
    and so $\chi_G^\mathbbm{1}(q,n) = \sum_{i=0}^n \qbinom{d+n-i}{d} a_i(q)$. 

We will now show that certain summands in this $q$-binomial expression are zero. 
    If $i > d$ then $\ascnum_G(\pi) +1 > d$ and so $\ascnum_G(\pi) > d-1$. 
    However, $\asc_G(\pi) \subseteq [d-1]$ so the terms with $i > d$ are zero. 
    If $i < \xi$ then $\ascnum_G(\pi) +1 < \xi$ and so the $G$-sequence would give a proper coloring of less
than $\xi$ colors, which contradicts Lemma~\ref{lem:independentgsequence}. So these terms are also zero. 
    Thus we may change indices to deduce
\begin{align*}
        \chi_G^\mathbbm{1}(q,n)
        &= \sum_{i=0}^n \qbinom{d+n-i}{d} a_i(q) 
         = \sum_{i=\xi}^d \qbinom{d+n-i}{d} a_i(q) 
         = \sum_{k=0}^{d-\xi} \qbinom{n+d - \xi -k}{d} a_{\xi+k}(q) \\
        &= \sum_{j=0}^{d-\xi} \qbinom{n+j}{d} a_{d-j}(q) \, .
\end{align*}    
Since permutations indexed by $a_{d-j}(q)$ have $\ascnum_G(\pi)+1 = d-j$ which implies $\ascnum_G(\pi) = d-1-j$, we must have $\desnum_G(\pi) = j$. This proves 
    \[
        \chi_G^\mathbbm{1}(q,n) = \sum_{j=0}^{d-\xi} \left( \qbinom{n+j}{d} \sum_{\substack{\pi \in S_d \\
\desnum_G(\pi) = j}} q^{ d+\sum_{j \in \asc_G(\pi)} d-j} \right). 
    \]
\end{proof}

\begin{remark}
    For a graph $G$ on $[d]$ with chromatic number $\xi$, we can express
    \[
        \sum_{\pi \in S_d} q^{\alpha_\pi} z^{\ascnum_G(\pi)+1} = a_\xi (q) z^\xi+ \cdots + a_d(q) z^d \in
\mathbb{Z}[q,z]
    \]
    where $\pi \in S_d$, $\alpha_\pi := d+\sum_{j \in \asc_G(\pi)} d-j$, and $a_i(q) := \sum q^{\alpha_\pi}$ where we sum over all permutations $\pi$ such that $\ascnum_G(\pi)+1 = i$, by the argument given in the proof of Corollary~\ref{cor:qchromaticqbinomial}.
\end{remark}

\begin{figure}
    \centering
    \begin{tikzpicture}
    \begin{scope}[xshift=0, yshift=0, scale=1]    
        \vertex[fill](a1) at (0,0) {};
        \vertex[fill](a2) at (1,0) {};
        \vertex[fill](a3) at (2,0) {};
        \vertex[fill](a4) at (3,0) {};

        \node[anchor=north] at (a1) {$1$};
        \node[anchor=north] at (a2) {$2$};
        \node[anchor=north] at (a3) {$3$};
        \node[anchor=north] at (a4) {$4$};

        \node[anchor=west] (b1) at (4,0) {\large $\frac{q^{10}z^4 + (3q^9 + 5q^8 + 3q^7)z^3 + (3q^7 + 5q^6 + 3q^5)z^2 + q^4z}{(1-z)(1-qz)(1-q^2z)(1-q^3z)(1-q^4z)}$};
        \node[] (b2) at (-1,0) {$G_1$};

    \end{scope}

    \begin{scope}[xshift=0, yshift=-40, scale=1]    
        \vertex[fill](a1) at (0,0) {};
        \vertex[fill](a2) at (1,0) {};
        \vertex[fill](a3) at (2,0) {};
        \vertex[fill](a4) at (3,0) {};

        \node[anchor=north] at (a1) {$1$};
        \node[anchor=north] at (a2) {$2$};
        \node[anchor=north] at (a3) {$3$};
        \node[anchor=north] at (a4) {$4$};

        \node[anchor=west] (b1) at (4,0) {\large $\frac{8q^{10}z^4 + (4q^9 + 6q^8 + 4q^7)z^3 + 2q^6z^2}{(1-z)(1-qz)(1-q^2z)(1-q^3z)(1-q^4z)}$};
        \node[] (b2) at (-1,0) {$G_2$};

        \draw[thick] (a1) --  (a2) -- (a3) -- (a4);
    \end{scope}

    \begin{scope}[xshift=0, yshift=-95, scale=1]    
        \vertex[fill](a1) at (1.5,0) {};
        \vertex[fill](a2) at (1.5,1) {};
        \vertex[fill](a3) at (1.5+0.86602,-0.5) {};
        \vertex[fill](a4) at (1.5-0.86602,-0.5) {};

        \node[anchor=north] at (a1) {$1$};
        \node[anchor=north west] at (a2) {$2$};
        \node[anchor=north] at (a3) {$3$};
        \node[anchor=north] at (a4) {$4$};

        \node[anchor=west] (b1) at (4,0) {\large $\frac{8q^{10}z^4 + (5q^9 + 4q^8 + 5q^7)z^3 + (q^7 + q^5)z^2}{(1-z)(1-qz)(1-q^2z)(1-q^3z)(1-q^4z)}$};
        \node[] (b2) at (-1,0) {$G_3$};

        \draw[thick] (a1) --  (a2);
        \draw[thick] (a1) --  (a3);
        \draw[thick] (a1) --  (a4);
    \end{scope}

    \begin{scope}[xshift=0, yshift=-150, scale=1]    
        \vertex[fill](a1) at (1.5-0.5,-0.5) {};
        \vertex[fill](a2) at (1.5-0.5,0.5) {};
        \vertex[fill](a3) at (1.5+0.5,-0.5) {};
        \vertex[fill](a4) at (1.5+0.5,0.5) {};

        \node[anchor=north] at (a1) {$1$};
        \node[anchor=north east] at (a2) {$2$};
        \node[anchor=north] at (a3) {$3$};
        \node[anchor=north west] at (a4) {$4$};

        \node[anchor=west] (b1) at (4,0) {\large $\frac{24q^{10}z^4}{(1-z)(1-qz)(1-q^2z)(1-q^3z)(1-q^4z)}$};
        \node[] (b2) at (-1,0) {$G_4$};

        \draw[thick] (a1) --  (a3) --(a4) -- (a2) -- (a1);
        \draw[thick] (a1) -- (a4);
        \draw[thick] (a2) -- (a3);
    \end{scope}

\end{tikzpicture}
    \caption{Graphs and their corresponding generating function of the $q$-chromatic polynomial with linear form $\mathbbm{1}$.}
    \label{fig:numeratorpolynomial}
\end{figure}
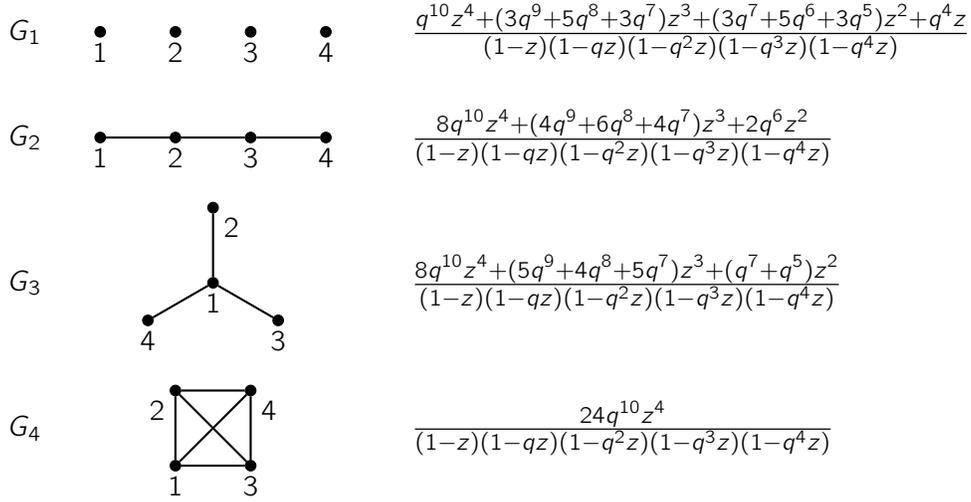

\begin{example}\label{ex:qchromaticgenfunction}
    We give the generating functions of the $q$-chromatic polynomials with linear form $\mathbbm{1}$ for the
graphs in Figure~\ref{fig:numeratorpolynomial}. 
    By Corollary~\ref{cor:qchromaticqbinomial}, we can express the $q$-chromatic polynomials as
\begin{align*}
        \chi_{G_1}^\mathbbm{1}(q,n) &= \qbinom{n+3}{4} q^4 + \qbinom{n+2}{4} (3q^7+5q^6+3q^5) + \qbinom{n+1}{4} (3 q^9 + 5 q^8+3 q^7) + \qbinom{n}{4} q^{10},
    \\
        \chi_{G_2}^\mathbbm{1}(q,n) &=\qbinom{n+2}{4} (2q^6) + \qbinom{n+1}{4} (4 q^9+6q^8+4q^7) + \qbinom{n}{4} (8 q^{10}),
    \\
        \chi_{G_3}^\mathbbm{1}(q,n)  &= \qbinom{n+2}{4} (q^7+q^5) + \qbinom{n+1}{4} (5q^9+4q^8+5q^7) + \qbinom{n}{4} (8 q^{10}),
    \\
        \chi_{G_4}^\mathbbm{1}(q,n)  &= \qbinom{n}{4} (24 q^{10}).
\end{align*} 
\end{example}

Corollary~\ref{cor:qchromaticqbinomial} gives a $q$-analogue of the above-mentioned result of
Chung--Graham~\cite{chunggrahamcoverpol} via setting $q=1$.
\begin{corollary}[Chung--Graham]\label{cor:chunggraham}
    For a graph $G$ on $d$ vertices,
    \[
        \chi_G(n) = \chi^\mathbbm{1}_G(1,n) = \sum_{j=0}^{d-\xi} \binom{n+j}{d} \left|\{ \pi \in S_d :
\desnum_G(\pi)=j\} \right|.
    \] 
\end{corollary}

\subsection{Connections to symmetric functions}\label{subsec:symmetricfunctions}
The integer point transform encodes integer points as exponents of monomials.
We could have instead encoded integer points as the indices of monomials; doing so for the integer points of
$K_G$ defines the chromatic symmetric function.
By partitioning $K_G$, we expressed the integer point transform $\sigma_{K_G}$ as a sum of rational functions in Theorem~\ref{thm:multivariablechromaticgenerating}.
Chow~\cite{chow1999} provided an analogous  result for the chromatic symmetric function using Stanley's  theory
of $P$-partitions and quasisymmetric functions~\cite{stanleychromatic}. 
For this section, we assume readers have some background on quasisymmetric functions but we review some important terms \cite[Section~7.19]{enumerativecombinatorics}.
Let $\mathcal{Q}$ denote the algebra of quasisymmetric functions.
For any subset $S \subseteq [d-1]$ the \Def{fundamental quasisymmetric function} is
\[
    F_{S,d}(x) := \sum_{\substack{{i_1 \leq \cdots \leq i_d} \\ i_j < i_{j+1} \text{ if } j \in S }} x_{i_1}
x_{i_2} \cdots x_{i_d} \, .
\]
These fundamental quasisymmetric functions form a basis for quasisymmetric functions and Chow wrote the
chromatic symmetric function in this basis~\cite[Corollary~1]{chow1999}.
\begin{theorem}[Chow]\label{thm:chowchromatic}
    Let $G$ be a graph on $[d]$, then
    \[
        X_G = \sum_{S \subseteq [d-1]} N_S F_{S,d}
    \]
    where $N_S$ is the number of permutations with $G$-ascent set $S$.
\end{theorem}
Since the $q$-chromatic polynomial with the linear form $\mathbbm{1}$ can be recovered from the principal specialization of the chromatic symmetric function, we can apply the principal specialization to Theorem~\ref{thm:chowchromatic} to deduce formulas for the $q$-chromatic polynomial.
In particular, we will see that the generating function for the principal specialization of order $m$ applied to Theorem~\ref{thm:chowchromatic} gives the same expression of the generating function we find for the $q$-chromatic polynomial with linear form $\mathbbm{1}$.
The \Def{principal specialization of order $m$} is the ring homomorphism $\ps_m: \mathcal{Q} \to \mathbb{Q}[q]$
given by
\[
    \ps_m(x_i) = \begin{cases}
        q^{i-1} & \text{if } 1 \leq i \leq m, \\
        0 & \text{otherwise.}
    \end{cases}
\]
Hence for a graph $G$ on $[d]$ and $n \in \mathbb{Z}_{>0}$,
\begin{align*}
    q^d \ps_n(X_G)
    &= q^d X_G(1,q,\ldots,q^{n-1},0,\ldots)
     = q^d \sum_{\substack{\text{proper colorings} \\ c:[d] \to [n]}} q^{\sum_{i \in [d]} c(i)-1}
     = \sum_{\substack{\text{proper colorings} \\ c:[d] \to [n]}} q^{\sum_{i \in [d]} c(i)} \\
    &= \chi_G^\mathbbm{1}(q,n) \, .
\end{align*}

Gessel gave the following expression for the generating function of the principal specialization of fundamental
quasisymmetric functions~\cite[Lemma 5.2]{gesseldescent}.
\begin{lemma}[Gessel]
    \[
        \sum_{m \geq 0} \ps_m(F_{S,d}) z^m =  \frac{z^{|S|+1} q^{\sum_{i \in S} d-i} }{(1-z)(1-qz) \cdots
(1-q^dz)} \, .
    \]
\end{lemma}

With Theorem~\ref{thm:chowchromatic}, we can find $q$-chromatic generating functions by applying the principal specialization:
\begin{align*}
    \sum_{n \geq 0} \chi_G^\mathbbm{1}(q,n) z^n &= \sum_{n \geq 0} q^d \ps_n(X_G) z^n \\
    &= q^d \sum_{n \geq 0} \ps_n \left( \sum_{S \subseteq [d-1]} N_S F_{S,d}\right) z^n \\
    &= q^d \sum_{S \subseteq [d-1]} N_S \sum_{n \geq 0} \ps_n(F_{S,d}) z^n \\
    &=  q^d \sum_{S \subseteq [d-1]} N_S \frac{z^{|S|+1} q^{\sum_{i \in S} d-i} }{(1-z)(1-qz) \cdots (1-q^dz)} \\
    &= \frac{\sum_{\pi \in S_d} q^{d+\sum_{i \in \asc_G(\pi)} d-i} z^{\ascnum_G(\pi)+1}}{(1-z)(1-qz) \cdots (1-q^dz)}
\end{align*}
where the last equality holds by definition of $N_S$. 
This is exactly the expression in Theorem~\ref{thm:qchromaticallones}.


\section{$G$-statistics and acyclic orientations}\label{sec:gstatsandorientations}

The $q$-chromatic polynomial was explored by Bajo et al.\ via the $q$-Ehrhart theory of order polytopes,
analogous to the non-$q$ setting.
This relates to Stanley's theory of $P$-partitions used by Chow to prove Theorem~\ref{thm:chowchromatic}.
We will see that these approaches can be unified via $G$-statistics.

Given a poset $Q = ([d],\preceq)$, its \Def{order polytope} is
\[
    \order{Q} = \left\{(x_1,\ldots,x_d) \in [0,1]^d : x_i \leq x_j \text{ if } i \preceq j \right\}.
\]
Let $\lambda \in \mathbb{Z}^d$ be some linear form and let $P \subseteq \mathbb{R}^d$ be a lattice polytope,
Chapoton~\cite{chapoton} defines
\[
    \ehr_P^\lambda(q,n) := \sum_{m \in nP \cap \mathbb{Z}^d} q^{\lambda \cdot m}.
\]
Bajo et al.\ use the following key lemma relating these $q$-Ehrhart polynomials of (open) order polytopes to the
$q$-chromatic polynomial~\cite[Lemma~3]{qchromatic}; all notation used in the sum is defined in the next subsection.
\begin{lemma}[Bajo et al]\label{lem:qchromaticqehrhart}
    Let $G$ be a graph on $[d]$ and let $\lambda \in \mathbb{Z}^d$ be a linear form. Then 
    \[
        \chi^\lambda_G(q,n) = \sum_{\rho \in \acyclic{G}} \ehr^\lambda_{\order{G_\rho}^\circ}(q,n+1) \, .
    \]
\end{lemma}

To compute the Ehrhart polynomials in the above sum, one considers acyclic orientations of $G$ and linear extensions of induced posets related to those acyclic orientations. 
Thus, to connect $G$-statistics to the work of Bajo et al., we need to establish a direct connection between permutations and pairs of acyclic orientations and linear extensions, which we do next.

\subsection{The main bijection}


Let $P$ be a poset on $[d]$. A \Def{linear extension} $\tau:P \to [d]$ is an order preserving bijection. 
A \Def{natural labeling} of a poset $P$ is some fixed linear extension $\omega:P \to [d]$ and we denote the labeled poset as $(P,\omega)$.
When given a natural labeling, we can identify a linear extension $\tau$ as a permutation given by the word $[\omega(\tau^{-1}(1))\,  \omega(\tau^{-1}(2)) \, \cdots \, \omega(\tau^{-1}(d))]$ of the labels of $(P,\omega)$. 
The collection of all such permutations is the \Def{Jordan--H\"older set}, denoted $\jhset{P,\omega}$ or
$\jhset{P}$ if the labeling is clear from the context.
For a poset $P$ with a natural labeling $\omega$ and linear extension $\tau$, we say that the linear extension has a \Def{descent} at $i \in [d-1]$ if for the associated word $\omega \circ \tau^{-1}$ we have $\omega \circ \tau^{-1}(i) > \omega \circ \tau^{-1} (i+1)$.
Denote the set of all descents by $\des(\omega \circ \tau^{-1})$.

\begin{example}\label{ex:posettheory}
    Let $P$ be the poset given in Figure~\ref{fig:posettheory}, then $\omega:P \to [4]$ where $\omega(1)=1$, $\omega(2)=3$, $\omega(3)=4$, and $\omega(4)=2$ is a natural labeling.
    We can also see this in Figure~\ref{fig:posettheory} since the labeled poset respects the usual order $1 \leq 2 \leq 3 \leq 4$.
    Moreover $\tau:P \to [4]$ defined by
    \[
        \tau(1) = 2, \ \tau(2)=4, \ \tau(3)=3, \ \tau(4)=1
    \]
    is a linear extension because $i \leq_P j$ for $i,j \in P$ implies $\tau(i) \leq \tau(j)$. We can think of the image of $\tau$ as defining the positions of elements of $P$ for a linear order that respects the relations of $P$.

    To see the corresponding word in $\jhset{P,\omega}$ for $\tau$ we have
    \[
        [\omega(\tau^{-1}(1))\omega(\tau^{-1}(2))\omega(\tau^{-1}(3))\omega(\tau^{-1}(4))] = [\omega(4)
\omega(1) \omega(3) \omega(2)] = [2143] \, .
    \]
    Note that $\tau^{-1}$ can be seen as a permutation of elements of $P$, and applying $\omega$ ensures we have a word of the labels of $P$.
    We also have $\des([2143])=\{1,3\}$.
\end{example}

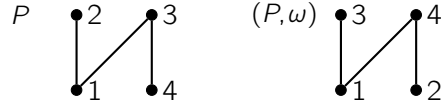
\begin{figure}
    \centering
    \begin{tikzpicture}
    \begin{scope}[xshift=0, yshift=0, scale=1]    
        \vertex[fill](a1) at (0,0) {};
        \vertex[fill](a2) at (0,1) {};
        \vertex[fill](a3) at (1,0) {};
        \vertex[fill](a4) at (1,1) {};
        
        \node[anchor=west] (b1) at (a1) {$1$};
        \node[anchor=west] (b2) at (a2) {$2$};
        \node[anchor=west] (b3) at (a3) {$4$};
        \node[anchor=west] (b4) at (a4) {$3$};
        \node[] (b5) at (-0.75,1) {\small $P$};

        \draw[thick] (a1) --  (a2);
        \draw[thick] (a1) -- (a4);
        \draw[thick] (a3) -- (a4);
      
    \end{scope}

    \begin{scope}[xshift=100, yshift=0, scale=1]    
        \vertex[fill](a1) at (0,0) {};
        \vertex[fill](a2) at (0,1) {};
        \vertex[fill](a3) at (1,0) {};
        \vertex[fill](a4) at (1,1) {};
        
        \node[anchor=west] (b1) at (a1) {$1$};
        \node[anchor=west] (b2) at (a2) {$3$};
        \node[anchor=west] (b3) at (a3) {$2$};
        \node[anchor=west] (b4) at (a4) {$4$};
        \node[] (b5) at (-0.75,1) {\small $(P,\omega)$};

        \draw[thick] (a1) --  (a2);
        \draw[thick] (a1) -- (a4);
        \draw[thick] (a3) -- (a4);
      
    \end{scope}
\end{tikzpicture}
    \caption{On the left, a poset $P$ with elements $[4]$. On the right, a natural labeling $\omega$ of $P$ given in Example~\ref{ex:posettheory}.}
    \label{fig:posettheory}
\end{figure}

Let $\acyclic{G}$ denote the set of acyclic orientations of a given graph $G$. For $\rho \in \acyclic{G}$, let
$G_\rho$ denote the \Def{induced poset} on elements $[d]$, were $i \to j$ in $\rho$ gives the relation $i \leq
j$ in $G_\rho$, followed by taking the transitive closure to produce a poset.
For a graph $G$ and $\rho \in \acyclic{G}$, let $\{v_1,\ldots,v_{i_1}\}$ be all minimal elements of $G_\rho$,
with $v_1 < \cdots < v_{i_1}$; next let $\{v_{i_1+1},\ldots,v_{i_2} \}$ be all the minimal elements after
removing $\{v_1,\ldots,v_{i_1}\}$ from $G_\rho$, with $v_{i_1+1} < \cdots < v_{i_2}$, and repeat until we have listed all elements of $G_\rho$. We define the \Def{rank labeling} $\eta_\rho:G_\rho \to [d]$ by $\eta_\rho(v_j)=j$.

\begin{figure}
    \centering
    \begin{tikzpicture}
    \begin{scope}[xshift=0, yshift=0, scale=1]    
        \vertex[fill](a1) at (0,0) {};
        \vertex[fill](a2) at (0,1) {};
        \vertex[fill](a3) at (0.95105,0.30901) {};
        \vertex[fill] (a4) at (0.58778,-0.80901) {};
        \vertex[fill] (a5) at (-0.58778,-0.80901) {};
        \vertex[fill](a6) at (-0.95105,0.30901) {};
        
        \node[anchor=north] (b1) at (a1) {$1$};
        \node[anchor=north west] (b2) at (a2) {$6$};
        \node[anchor=north] (b3) at (a3) {$5$};
        \node[anchor=north] (b4) at (a4) {$4$};
        \node[anchor=north] (b5) at (a5) {$3$};
        \node[anchor=north] (b6) at (a6) {$2$};
        \node (c1) at (0,-1.5) {$G$};

        \draw[thick] (a1) --  (a2);
        \draw[thick] (a1) --  (a3);
        \draw[thick] (a1) --  (a4);
        \draw[thick] (a1) --  (a5);
        \draw[thick] (a1) --  (a6);
      
    \end{scope}

    \begin{scope}[xshift=90, yshift=0, scale=1]    
        \vertex[fill](a1) at (0,0) {};
        \vertex[fill](a2) at (0,1) {};
        \vertex[fill](a3) at (0.95105,0.30901) {};
        \vertex[fill] (a4) at (0.58778,-0.80901) {};
        \vertex[fill] (a5) at (-0.58778,-0.80901) {};
        \vertex[fill](a6) at (-0.95105,0.30901) {};
        
        \node[anchor=north] (b1) at (a1) {$1$};
        \node[anchor=north west] (b2) at (a2) {$6$};
        \node[anchor=north] (b3) at (a3) {$5$};
        \node[anchor=north] (b4) at (a4) {$4$};
        \node[anchor=north] (b5) at (a5) {$3$};
        \node[anchor=north] (b6) at (a6) {$2$};
        \node (c1) at (0,-1.5) {$\rho$};

        \draw[thick,mid arrow={stealth[reversed]}] (a1) --  (a2);
        \draw[thick,mid arrow={stealth[reversed]}] (a1) --  (a3);
        \draw[thick,mid arrow={stealth}] (a1)--(a4);
        \draw[thick,mid arrow={stealth}] (a1)--(a5);
        \draw[thick,mid arrow={stealth}] (a1)--(a6);
      
    \end{scope}

    \begin{scope}[xshift=175, yshift=0, scale=1]    
        \vertex[fill] (a1) at (-0.5,-0.75) {};
        \vertex[fill] (a2) at (0.5,-0.75) {};
        \vertex[fill](a3) at (0,0) {};
        \vertex[fill](a4) at (0.75,0.75) {};
        \vertex[fill](a5) at (0,0.75) {};
        \vertex[fill](a6) at (-0.75,0.75) {};

        \node[anchor=west] (b1) at (a1) {$6$};
        \node[anchor=west] (b2) at (a2) {$5$};
        \node[anchor=west] (b3) at (a3) {$1$};
        \node[anchor=west] (b4) at (a4) {$2$};
        \node[anchor=west] (b5) at (a5) {$3$};
        \node[anchor=west] (b6) at (a6) {$4$};
        \node (c1) at (0,-1.5) {$G_\rho$};

        \draw[thick] (a1) --  (a3);
        \draw[thick] (a2) -- (a3);
        \draw[thick] (a3) -- (a4);
        \draw[thick] (a3) -- (a5);
        \draw[thick] (a3) -- (a6);
      
    \end{scope}

    \begin{scope}[xshift=255, yshift=0, scale=1]    
        \vertex[fill] (a1) at (-0.5,-0.75) {};
        \vertex[fill] (a2) at (0.5,-0.75) {};
        \vertex[fill](a3) at (0,0) {};
        \vertex[fill](a4) at (0.75,0.75) {};
        \vertex[fill](a5) at (0,0.75) {};
        \vertex[fill](a6) at (-0.75,0.75) {};

        \node[anchor=west] (b1) at (a1) {$2$};
        \node[anchor=west] (b2) at (a2) {$1$};
        \node[anchor=west] (b3) at (a3) {$3$};
        \node[anchor=west] (b4) at (a4) {$4$};
        \node[anchor=west] (b5) at (a5) {$5$};
        \node[anchor=west] (b6) at (a6) {$6$};
        \node (c1) at (0,-1.5) {$(G_\rho,\eta_\rho)$};

        \draw[thick] (a1) --  (a3);
        \draw[thick] (a2) -- (a3);
        \draw[thick] (a3) -- (a4);
        \draw[thick] (a3) -- (a5);
        \draw[thick] (a3) -- (a6);
      
    \end{scope}

\end{tikzpicture}
    \caption{A graph, an acyclic orientation, the corresponding induced poset, and the induced poset with the
rank labeling.}
    \label{fig:inducedposet}
\end{figure}
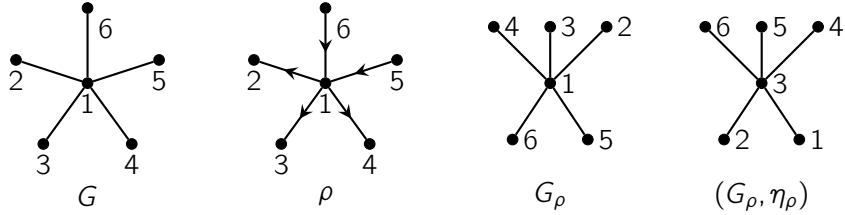

\begin{example}
    Figure~\ref{fig:inducedposet} gives an example of an induced poset for the star graph on $5$ vertices for
some acyclic orientation. For the induced poset, we have the rank labeling $\eta_\rho$ where 
    \[
        \eta_\rho(1)=3, \ \eta_\rho(2)=4, \ \eta_\rho(3)=5, \ \eta_\rho(4)=6, \ \eta_\rho(5)=1, \ \eta_\rho(6)=2.
    \]
\end{example}

The following lemma will play a key role in our arguments in this section.
\begin{lemma}\label{lem:edgerank}
    Let $G$ be a graph on $[d]$ and $\pi \in S_d$. If $\pi(i) \pi(j) \in E(G)$ and $i <j$, then $\rk(\pi(i))<\rk(\pi(j))$.
\end{lemma}
\begin{proof}
    When computing the rank $\rk(\pi(j))$, consider a path defining $\rk(\pi(i))$ along with the edge $\pi(i)\pi(j)$. 
    In other words, we know for $r := \rk(\pi(i))$ there exist positions $i_1 < i_2 < \dots < i_r = i$ such that $\{\pi(i_j), \pi(i_{j+1}) \} \in E$ for $1 \le j < r$.
    Thus $\pi(i_1) \pi(i_2) \ldots \pi(i) \pi(j)$ is a possible path satisfying the definition in rank for $\pi(j)$. 
    Since $\rk(\pi(j))$ is defined from the longest such path satisfying this condition, 
    \[
        \rk(\pi(j)) \geq \rk(\pi(i))+1 > \rk(\pi(i)) \, . \qedhere 
    \] 
\end{proof}

Let $G$ be a graph on $[d]$. 
For $\pi \in S_d$ we define the \Def{$\pi$-rank acyclic orientation}, denoted $\rho_\pi$, to be the orientation where 
\[
    \rho_\pi \text{ orients the edge } ij \in E(G) \text{ by } i \to j \text{ whenever } \rk(i) < \rk(j).
\]
To show $\rho_\pi$ is well defined, consider two adjacent vertices $v_1,v_2$. 
Then there exist $i_1,i_2$ such that $\pi(i_1) = v_1, \pi(i_2) = v_2$; without loss of generality say $i_1 < i_2$. 
Then $\rk(v_1) < \rk(v_2)$ by Lemma~\ref{lem:edgerank}.
To show that this orientation is acyclic, note that on a directed path we must have ranks be strictly increasing and so having a cycle would contradict these strict inequalities of ranks.
Given a natural labeling $\omega$ of $G_{\rho_\pi}$, we also define 
\[
    \sigma_\pi := \omega \circ \pi = \omega \circ (\pi^{-1})^{-1}
\]
which we can identify as the word $[\omega(\pi(1)) \omega(\pi(2)) \cdots \omega(\pi(d))]$.

\begin{example}
    Consider the bowtie graph $G$ from Figure~\ref{fig:gsequence} and $\pi = [31254]$. Recall from Example~\ref{ex:gsequence} that the rank function $\rk$ gives $\rk(3) = 1$, $\rk(1)=1$, $\rk(2)=2$, $\rk(5)=3$, and $\rk(4)=4$. The $\pi$-rank acyclic orientation is shown in Figure~\ref{fig:pirankposet}. Moreover note that the rank labeling is a natural labeling for $G_{\rho_\pi}$ and so for this labeling $\sigma_\pi = \eta_\rho \circ \pi$ as a word is $[21354]$.
\end{example}

\begin{figure}
    \centering
    \begin{tikzpicture}
    \begin{scope}[xshift=0, yshift=0, scale=1]    
        \vertex[fill](a1) at (0,0) {};
        \vertex[fill](a2) at (1,0.5) {};
        \vertex[fill](a3) at (1,-0.5) {};
        \vertex[fill](a4) at (-1,0.5) {};
        \vertex[fill](a5) at (-1,-0.5) {};

        \node[anchor=south] (b1) at (a1) {$2$};
        \node[anchor=south] (b2) at (a2) {$1$};
        \node[anchor=north] (b3) at (a3) {$4$};
        \node[anchor=south] (b4) at (a4) {$5$};
        \node[anchor=north] (b5) at (a5) {$3$};
        \node (c1) at (0,-1.25) {$\rho_\pi$};

        \draw[thick, mid arrow={stealth[reversed]}] (a1) -- (a2);
        \draw[thick, mid arrow={stealth}] (a1) -- (a3);
        \draw[thick, mid arrow={stealth[reversed]}] (a1) -- (a5);
        \draw[thick, mid arrow={stealth}] (a1) -- (a4);
        \draw[thick, mid arrow={stealth}] (a2) -- (a3);
        \draw[thick, mid arrow={stealth[reversed]}] (a4) -- (a5);
    \end{scope}

    \begin{scope}[xshift=90, yshift=0, scale=1]    
        \vertex[fill](a1) at (0.666,-0.666) {};
        \vertex[fill](a2) at (-0.666,-0.666) {};
        \vertex[fill](a3) at (0,0) {};
        \vertex[fill](a4) at (0.666,0.666) {};
        \vertex[fill](a5) at (-0.666,0.666) {};

        \node[anchor=west] (b1) at (a1) {$1$};
        \node[anchor=west] (b2) at (a2) {$3$};
        \node[anchor=west] (b3) at (a3) {$2$};
        \node[anchor=west] (b4) at (a4) {$4$};
        \node[anchor=west] (b5) at (a5) {$5$};
        \node (c1) at (0,-1.25) {$G_{\rho_\pi}$};
        
        \draw[thick] (a1) -- (a3);
        \draw[thick] (a2) -- (a3);
        \draw[thick] (a3) -- (a4);
        \draw[thick] (a3) -- (a5);
    \end{scope}

    \begin{scope}[xshift=175, yshift=0, scale=1]    
        \vertex[fill](a1) at (0.666,-0.666) {};
        \vertex[fill](a2) at (-0.666,-0.666) {};
        \vertex[fill](a3) at (0,0) {};
        \vertex[fill](a4) at (0.666,0.666) {};
        \vertex[fill](a5) at (-0.666,0.666) {};

        \node[anchor=west] (b1) at (a1) {$1$};
        \node[anchor=west] (b2) at (a2) {$2$};
        \node[anchor=west] (b3) at (a3) {$3$};
        \node[anchor=west] (b4) at (a4) {$4$};
        \node[anchor=west] (b5) at (a5) {$5$};
        \node (c1) at (0,-1.25) {$(G_{\rho_\pi},\eta_{\rho_\pi})$};
        
        \draw[thick] (a1) -- (a3);
        \draw[thick] (a2) -- (a3);
        \draw[thick] (a3) -- (a4);
        \draw[thick] (a3) -- (a5);
    \end{scope}

\end{tikzpicture}
    \caption{The $\pi$-rank acyclic orientation $\rho_\pi$ of the bowtie graph in Figure~\ref{fig:gsequence} for $\pi=[31254]$, the induced poset $G_{\rho_\pi}$, and the induced poset with the rank labeling $(G_{\rho_\pi},\eta_{\rho_\pi})$.}
    \label{fig:pirankposet}
\end{figure}
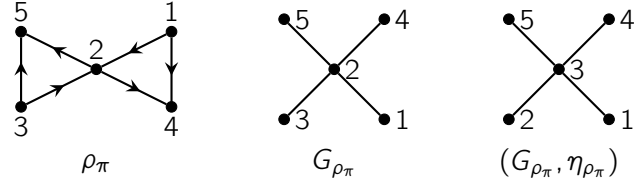

Theorem~\ref{thm:bijectiongstats} below exhibits that the rank functions for $G$-statistics encode acyclic orientations, i.e., the induced posets of $G$, by showing that each permutation in $S_d$ corresponds uniquely to a linear extension of an induced poset.
Our proof below closely mirrors a proof due to Chow \cite{chow1999} arising in the work discussed previously regarding quasisymmetric functions and the chromatic symmetric function.
Specifically, Chow used Stanley's \cite{stanleychromatic} theory of $P$-partition functions to establish a bijection between permutations and pairs of acyclic orientations and linear extensions, where each acyclic orientation is given an order-reversing labeling of the induced poset.
He also gave a labeling of these induced posets to show that the usual descent set of a permutation is equal to the descent set of the corresponding linear extension under this bijection and labeling.
Since he used $P$-partitions, i.e., order-reversing maps from a poset $P$ to $\mathbb{R}_{\geq 0}$, his results
have direct analogues to order-preserving maps of posets, which we use for order polytopes in Theorem~\ref{thm:bijectiongstats}. 

Further, Chow assigned the acyclic orientation in a manner that does not involve rank functions, though the resulting orientations are in fact equivalent.
Translating Chow's work to our context of order-preserving maps, he considers the following: for $G$ a graph on $[d]$, a permutation $\pi \in S_d$, and an edge $\pi(i)\pi(j) \in E(G)$, orient $\pi(i) \to \pi(j)$ for $\rho_\pi$ if $i <j$.
The resulting acyclic orientation is identical to our $\pi$-rank acyclic orientation defined above since for $G$ a graph on $[d]$, $\pi \in S_d$, and $\pi(i)\pi(j) \in E(G)$, we have $\rk(\pi(i)) < \rk(\pi(j))$ if and only if $i <j$. 
The backwards direction of this equivalence is Lemma~\ref{lem:edgerank}, while the forward direction can also be proven with Lemma~\ref{lem:edgerank} by contrapositive. 

Thus, while the structure of the proof of Theorem~\ref{thm:bijectiongstats} closely mirrors Chow's work, to set
this in the context considered in this paper, we provide a complete proof of the bijection and apply it to our setting of order-preserving maps.

\begin{theorem}\label{thm:bijectiongstats}
    For a graph $G$ on $[d]$, for each $\rho \in \acyclic{G}$ let $\omega_\rho$ be  a natural labeling of
$G_\rho$. Then the map
    \[
        \varphi: S_d \to\{(\rho,\sigma) \mid \rho \in \mathcal{A}(G), \sigma \in \mathcal{L}(G_\rho,\omega_\rho) \}
    \]
    defined by
    \[
    \pi \mapsto (\rho_\pi, \sigma_\pi)
    \]
is a bijection.
    Moreover, if we use the rank labelings for our induced posets with $\sigma_\pi = \eta_{\rho_\pi} \circ \pi$, then $\des_G(\pi) = \des(\sigma_\pi)$ where $\des$ is the usual descent statistic.
\end{theorem}

There are two key observations to highlight in the theorem above.
First, while multiple permutations might induce the same permutation-rank acyclic orientation $\rho$, hence the same induced poset $G_\rho$, the choice of $\omega_\rho$ depends only on $\rho$ and is independent of which permutation is being considered.
As a result of this, the second observation is that each permutation which induces $\rho$ via its rank function will yield a unique linear extension of $G_\rho$ with respect to $\omega_\rho$; it is this pair of an acyclic orientation and a linear extension that yields the bijective correspondence with permutations.

\begin{proof}
    We will first prove bijectivity and then analyze the descent sets in the case where we use the rank labelings on each induced poset. 
    Suppose $\pi \in S_d$. Recall $\rho_\pi$ is an acyclic orientation of the edges $i \to j$ whenever $\rk(i) < \rk(j)$. Also recall $\sigma_\pi := \omega_{\rho_\pi} \circ \pi$, where $\omega_{\rho_\pi}$ is the same for all $\pi$ and $\pi'$ such that $\rho_\pi=\rho_{\pi'}$.
    
    We claim that $\sigma_\pi \in \mathcal{L}(G_{\rho_\pi})$. 
    To prove this, we have that $\sigma_\pi = \omega_{\rho_\pi} \circ (\pi^{-1})^{-1}$, where $\pi^{-1}$ is defined as $\pi$ is a bijection.
    We can consider the domain of $\pi^{-1}$ to be the elements of the induced poset $G_{\rho_\pi}$ because these elements are given by the set $[d]$.
    This means we want to show $\pi^{-1}:G_{\rho_\pi} \to [d]$ is a linear extension, i.e., an order preserving bijection from $G_{\rho_\pi}$ to $[d]$.
    In other words, for the elements, not the labeling, of $G_{\rho_\pi}$ we want to show that $\pi^{-1}$ respects the relations of this poset.
    We know $\pi^{-1}:G_{\rho_\pi} \to [d]$ is a bijection, so all that is left to show is that it is order preserving.
    To do so, suppose $i \prec j$ in $G_{\rho_\pi}$ is a cover relation of elements; hence it must form an edge in $G$ and $i \neq j$. 
    Denote $v_i := \pi^{-1}(i)$ and $v_j:=\pi^{-1}(j) \in [d]$. 
    For contradiction suppose $v_j < v_i$. By Lemma~\ref{lem:edgerank},
    \[
        \rk(j)= \rk(\pi(v_j))< \rk(\pi(v_i)) = \rk(i).
    \]
    This inequality contradicts that $i \preceq j$ in $G_{\rho_\pi}$ since this relation implies $\rk(i) \leq \rk(j)$. 
    So we must have $v_i \leq v_j$, in other words $\pi^{-1}$ preserves cover relations. 
    This means $\pi^{-1}$ is order preserving, hence by definition $\sigma_\pi = \omega_\rho \circ \pi \in \mathcal{L}(G_{\rho_\pi})$.

    To define the inverse map $\psi:\{(\rho,\sigma) \mid \rho \in \mathcal{A}(G), \sigma \in \mathcal{L}(G_\rho)\} \to S_d$ 
    consider $(\rho,\sigma)$ such that $\rho \in \acyclic{G}$ and $\sigma \in \jhset{G_\rho}$. 
    By definition, $\sigma = \omega_\rho \circ \tau^{-1}$ for some order preserving map $\tau:G_\rho \to [d]$. We send $\psi(\rho,\sigma)= \tau^{-1}$. 
    We will now verify that $\psi$ is the inverse of $\varphi$. First, $\psi \circ \varphi (\pi) = \psi(\rho_\pi,\sigma_\pi) = \pi$ since $\sigma_\pi = \omega_\rho \circ \pi$. 
    We next want to show $\varphi \circ \psi (\rho,\sigma) = (\rho,\sigma)$ where $\sigma = \omega_\rho \circ \tau^{-1}$.
    By definition of $\psi$, we have $\varphi(\psi(\rho,\sigma))=\varphi(\tau^{-1})$. 
    For $\varphi(\tau^{-1})$, we have $\sigma_{\tau^{-1}} = \omega_\rho \circ \tau^{-1} = \sigma$.
    We claim $\rho_{\tau^{-1}} = \rho$. 
    By definition of $\rho_{\tau^{-1}}$, we have $i \to j$ in $\rho_{\tau^{-1}}$ whenever $\rk[\tau^{-1}](i) < \rk[\tau^{-1}](j)$
    so it suffices to show that for an edge $ij \in E(G)$ we have $i \leq j$ in $G_\rho$ if and only if  $\rk[\tau^{-1}](i) < \rk[\tau^{-1}](j)$.
    For the forward direction, suppose $i \preceq j \in G_\rho$, then $\tau(i) \leq \tau(j)$ since $\tau$ is order preserving.
    Hence
    \[
        \rk[\tau^{-1}](i) = \rk[\tau^{-1}](\tau^{-1} (\tau(i))) < \rk[\tau^{-1}](\tau^{-1} (\tau(j))) = \rk[\tau^{-1}](j)
    \]
    by Lemma~\ref{lem:edgerank}.
    For the backward direction, suppose $\rk[\tau^{-1}](i) < \rk[\tau^{-1}](j)$ then we claim $i \preceq j$ in
    $G_\rho$, since otherwise $\tau(i) > \tau(j)$ and by the same argument as the forward direction
    $\rk[\tau^{-1}](i) > \rk[\tau^{-1}](j)$, giving a contradiction.
    This shows $\rho_{\tau^{-1}} = \rho$ and proves that $\varphi$ is a bijection.

    Next, when for each $\rho \in \acyclic{G}$ we have that $\eta_\rho$ is the rank labeling, we will show $\des_G(\pi) = \des(\sigma_\pi)$ for every permutation $\pi \in S_d$. 
    For a permutation $\pi \in S_d$, we know the rank labeling of $G_{\rho_\pi}$ does the following. 
    Say $\{v_1, \ldots, v_{i_1}\}$ are the vertices of $G$ minimal in the poset, where $v_1< \cdots < v_{i_1}$,
    $\{v_{i_1 + 1}, \ldots, v_{i_2}\}$ are all vertices of $G$ minimal in $G\setminus\{v_1, \ldots, v_{i_1}\}$, where $v_{i_1 + 1} < \cdots < v_{i_2}$, and so on, then set $\eta_\rho(v_j) = j$. 
    We claim $\eta_{\rho_\pi}$ is a natural labelling since $i \preceq j$ in $G_{\rho_\pi}$ implies $\rk(i) \leq \rk(j)$ which must mean $\eta_\rho(i) \leq \eta_\rho(j)$ by construction of our labeling. 
    Additionally, any permutation with the same rank function $\rk$ gives the same acyclic orientation and so $\eta_{\rho_\pi}$ depends only on this acyclic orientation.
    We now show the descent statistics are the same under $\eta_{\rho_\pi}$. 
    By construction of our labeling, $\eta_\rho(i) < \eta_\rho(j)$  if and only if
    \[
        \rk(i) < \rk(j) \quad \text{or} \quad \rk(i) = \rk(j) \text{ and } i < j \, .
    \]
    Thus, $i \in \des_G(\pi)$ if and only if $\rk(\pi(i)) > \rk(\pi(i+1))$ or $\rk(\pi(i)) = \rk(\pi(i+1)) \text{ and } \pi(i) > \pi(i+1)$. 
    By our prior observation, this condition is equivalent to $\eta_\rho(\pi(i)) > \eta_\rho(\pi(i+1))$ which holds if and only if $i \in \des( \eta_\rho \circ \pi) = \des(\sigma_\pi)$.
\end{proof}
 
Viewing the rank function as defining an acyclic orientation motivates the definitions for $G$-statistics. 
Using the rank function of a permutation to produce a poset makes it also easier to see that the rank labeling is a natural labeling. 
In addition, Theorem~\ref{thm:bijectiongstats} shows that for any $\rho \in \acyclic{G}$, the rank labeling of $G_\rho$ ensures that the usual descent set of a linear extension is equal to the $G$-descent set of the corresponding permutation. 

We can apply the same arguments by Bajo et al.\ for order polytopes to give the generating function of the $q$-chromatic polynomial using $q$-Ehrhart theory and linear extensions of induced posets.
Bajo et al.\ give the generating function for $G$-partitions using linear extensions of induced posets.
We can apply the connection they found in Lemma~\ref{lem:qchromaticqehrhart} to obtain the generating function for the $q$-chromatic polynomial.
We will only sketch the details for the $q$-chromatic polynomial with linear form $\mathbbm{1} \in \mathbb{Z}^d$, as this case has a nice numerator polynomial in the rational expression of the generating function.
In Section~\ref{sec:symmetry}, we will see more properties of this polynomial.

\begin{theorem}\label{thm:qchromaticgenfunctiondoublesum}
    Let $G = ([d],E)$ be a graph, and for $\rho \in \acyclic{G}$, let $\omega_\rho$ be any natural labeling of $G_\rho$. Then
    \[
        \frac{\sum_{(\rho,\sigma)} q^{d+\sum_{j \in \asc(\sigma)} d-j} z^{\ascnum(\sigma)+1}}{(1-z)(1-qz) \cdots (1-q^d z)} = \frac{\sum_{\pi \in S_d} q^{d+\sum_{j \in \asc_G(\pi)} d-j} z^{\ascnum_G(\pi)+1}}{(1-z)(1-qz) \cdots (1-q^d z)}
    \]
    where $(\rho,\sigma)$ ranges over $\rho \in \acyclic{G}$ and $\sigma \in \jhset{G_\rho}$ and $\asc$ is the
usual ascent statistic for $(G_\rho,\omega_\rho)$.
\end{theorem}
\begin{proof}
    Recall that Bajo et al.\ proved Lemma~\ref{lem:qchromaticqehrhart}, which states
    \[
        \sum_{n \geq 0} \chi^\mathbbm{1}_G(q,n) z^n = \sum_{\rho \in \acyclic{G}} \sum_{n \geq 0}  \ehr^\mathbbm{1}_{\order{G_\rho}^\circ}(q,n+1) z^n.
    \]
    We can find this generating function via the integer point transform $\sigma_{K_{\widehat{G_\rho}}^\circ}(z_1,\ldots,z_{d+1})$, where $K_{\widehat{G_\rho}}^\circ$ is the homogenization of the open order polytope of $G_\rho$ with the last coordinate corresponding to height, using descent statistics \cite[Exercise 6.23]{crt}.
    For each $\rho \in \acyclic{G}$, the integer point transform is
    \[
    \sigma_{K_{\widehat{G_\rho}}^\circ}(z_1,\ldots,z_{d+1}) = \sum_{\tau \in \jhset{G_\rho}} \frac{z_1 \cdots z_d z_{d+1}^2 \prod_{j \in \asc(\tau)} z_{\tau(j+1)} \cdots z_{d+1} }{(1-z_{d+1}) \prod_{i=1}^{d}(1-z_{\tau(i)} \cdots z_{d+1}) }  \, .
    \]
    Specializing this integer point transform to $z_i = q$ for all $i \in [d]$ and $z_{d+1}=z$ yields
    \[
        \sigma_{K_{\widehat{G_\rho}}^\circ}(q,\ldots,q,z) = \sum_{n \geq 1}  \ehr^\mathbbm{1}_{\order{G_\rho}^\circ}(q,n) z^n.
    \]
    Thus
    \begin{align*}
        \sum_{n \geq 0} \chi^\mathbbm{1}_G(q,n) z^n &= \sum_{\rho \in \acyclic{G}} \sum_{n \geq 0} \ehr^\mathbbm{1}_{\order{G_\rho}^\circ}(q,n+1) z^n \\
        &= \sum_{\rho \in \acyclic{G}} z^{-1} \sum_{n \geq 1}  \ehr^\mathbbm{1}_{\order{G_\rho}^\circ}(q,n) z^n \\
        &= \sum_{\rho \in \acyclic{G}} z^{-1} \sigma_{K_{\widehat{G_\rho}}^\circ}(q,\ldots,q,z) \\
        &= \sum_{\rho \in \acyclic{G}} z^{-1}\sum_{\tau \in \jhset{G_\rho,\omega_\rho}} \frac{q^{d+\sum_{j \in \asc(\tau)}d-j} z^{\ascnum(\tau)+2}}{(1-z)(1-qz) \cdots(1-q^dz)} \\
        &= \frac{\sum_{\rho \in \acyclic{G}} \sum_{\tau \in \jhset{G_\rho}} q^{d+\sum_{j \in \asc(\tau)}d-j} z^{\ascnum(\tau)+1} }{(1-z)(1-qz) \cdots(1-q^dz)}.
    \end{align*} 
    By Theorem~\ref{thm:qchromaticallones}, we know that $\sum_{n \geq 0} \chi^\mathbbm{1}_G(q,n) z^n = \frac{\sum_{\pi \in S_d} q^{d+\sum_{j \in \asc_G(\pi)} d-j} z^{\ascnum_G(\pi)+1}}{(1-z)(1-qz) \cdots (1-q^d z)}$ and so the two rational functions in the statement of the theorem are equal.
\end{proof}

As a consequence of this rational function equality, the numerator polynomial of the $q$-chromatic polynomial with the linear form $\mathbbm{1}$ can be computed using any natural labeling on the induced posets, and so $G$-statistics unify these expressions.

\begin{corollary}\label{cor:linearextensionnumerator}
    Let $G = ([d],E)$ be a graph, and for $\rho \in \acyclic{G}$, let $\omega_\rho$ be any natural labeling of $G_\rho$. Then
    \[
        \sum_{\pi \in S_d} q^{\sum_{j \in \asc_G(\pi)} d-j} z^{\ascnum_G(\pi)}= \sum_{(\rho,\sigma)} q^{\sum_{j \in \asc(\sigma)} d-j} z^{\ascnum(\sigma)}
    \]
    where $(\rho,\sigma)$ ranges over $\rho \in \acyclic{G}$ and $\sigma \in \jhset{G_\rho}$ and $\asc$ is the usual ascent statistic for $(G_\rho,\omega_\rho)$.
\end{corollary}

\subsection{The leading coefficient of $\widetilde{\chi}_G^\mathbbm{1}(q,x)$}\label{subsec:gmajorpoly}

Bajo et al.~\cite{qchromatic} showed that $\chi_G^\lambda(q,n)$ is a polynomial
in $[n]_q := 1 + q + \dots + q^{ n-1 } $ with coefficients in $\mathbb{Z}(q)$.
We denote this polynomial by $\widetilde{\chi}_G^\mathbbm{1}(q,x) \in \mathbb{Q}(q)[x]$.
They show that the leading coefficient with respect to $x$ is the polynomial $\frac{1}{[d]_q!} \sum_{(\rho,\sigma)} q^{d +\maj(\sigma)}$, where we sum over all $\rho \in \acyclic{G}$ and all $\sigma \in \jhset{G_\rho}$, where $G_\rho$ is given some natural labeling \cite{qchromatic}. 

\begin{example}
    Let $G$ be the path graph on two vertices. 
    Bajo et al.~\cite{qchromatic} computed  $\widetilde{\chi}_G^\mathbbm{1}(q,x) = \frac{2 q^2}{1+q} x^2 + \frac{-2q^2}{1+q}x$. 
    The leading coefficient is $\frac{2 q^2}{1+q} = \frac{2q^2}{[2]_q} $.
    Indeed, $2 = \sum_{(\rho,\sigma)} q^{\maj(\sigma)}$ since we only have two acyclic orientations and each has only one linear extension with no descents.
\end{example}

The following corollary shows how Theorem~\ref{thm:bijectiongstats} provides an alternative form of the leading term of $\widetilde{\chi}_G^\mathbbm{1}(q,x)$ involving $G$-major indices.
\begin{corollary}\label{cor:leadingcoef}
    The leading coefficient of $\tilde{\chi}_G^\mathbbm{1}(q,x)$ is 
    \[
        \frac{1}{[d]_q!} \sum_{(\rho,\sigma)} q^{d + \maj(\sigma)} = \frac{1}{[d]_q!}\sum_{\pi \in S_d}
q^{d+\maj_G(\pi)} .
    \]
\end{corollary}
\begin{proof}
    By our bijection in Theorem~\ref{thm:bijectiongstats}, each $(\rho,\sigma)$ corresponds uniquely to a permutation where the descents sets are the same under the rank labeling. Hence the $G$-major index and usual major index must be the same. 
\end{proof}

For a graph $G$ on $[d]$, let the \Def{$G$-major index polynomial} be 
\[
    \sum_{\pi \in S_d} q^{\maj_G(\pi)}.
\]
We will now show that the degree of the $G$-major index polynomial connects the minimum sum coloring problem to the leading coefficient of $\widetilde{\chi}_G^\mathbbm{1}(q,x)$.
The minimum sum coloring problem was first studied using graph theory by Ewa Kubicka \cite{chromaticsums}.

\begin{lemma}\label{lem:minimumandcolorings}
    We have
    \[
        \min\left\{ d+\sum_{j \in \asc_G(\pi)} d-j : \pi \in S_d \right\} = \min\left\{ \sum_{v \in V} f(v): f:V \to [m] \text{ a coloring} \right\}
    \]
    and the number of colorings achieving the minimum sum of colors is the same as the number of permutations
achieving the minimum on the left-hand side. 
\end{lemma}
\begin{proof}
    Minimizing the sum of colors is equivalent to minimizing $\mathbbm{1} \cdot x$ for $x \in K_G \cap \mathbb{Z}_{>0}^d$. In addition, the minimum of $\mathbbm{1} \cdot x$ for $x \in \Delta_\pi \cap \mathbb{Z}_{>0}^d$ for $\pi \in S_d$ must occur at the integer point from the $G$-sequence coloring.
    In other words, suppose $(S_1,\ldots,S_k)$ is the $G$-sequence of $\pi$, then let $x_i = 1$ for $i \in S_1$,
$x_j=2$ where $j \in S_2$ and so on. Since any other coloring in $\Delta_\pi$ either splits our $G$-sequence
into more sets or uses the same sets of the $G$-sequence with higher colors, the $G$-sequence coloring is a
minimum in each cone. 
Moreover, $d + \sum_{j \in \asc_G(\pi)} d-j = \sum_{i=1}^k i |S_i|$ by Proposition~\ref{prop:weightedcolorsum}. 
Thus, taking the minimum value over all cones gives $\min\left\{ d+\sum_{j
\in \asc_G(\pi)} d-j : \pi \in S_d \right\}$ which must equal $\min\left\{ \sum_{v \in V} f(v): f:V \to [m]
\text{ a coloring} \right\}$. Moreover, each coloring achieving this minimum must be in some unique cone, by our
construction of the subdivision of $K_G$. 
So, the number of minimum colorings is also the same as the number of permutations achieving $\min\left\{ d+\sum_{j \in \asc_G(\pi)} d-j : \pi \in S_d \right\}$.
\end{proof}

\begin{corollary}\label{cor:gmajordegree}
    For a graph $G=([d],E)$, the degree of $\sum_{\pi \in S_d} q^{\maj_G(\pi)}$ equals
    \[
        \binom{d+1}{2} - \min\left\{ \sum_{v \in V} f(v): f:V \to [m] \text{ a coloring} \right\}
    \]
    and its leading coefficient equals the number of colorings achieving this minimum with respect to the sum of colors.
\end{corollary}
\begin{proof}
    We will reformulate the powers of $q$ to apply Lemma~\ref{lem:minimumandcolorings}.
    First, we know for $\pi \in S_d$
    \[
        \maj_G(\pi) + \sum_{j \in \asc_G(\pi)} j =  \sum_{i \in \des_G(\pi)} i + \sum_{j \in \asc_G(\pi)} j =
\sum_{i \in [d-1]} i = \binom{d}{2} \, ,
    \]
    whence $\maj_G(\pi) = \binom{d}{2} - \sum_{j \in \asc_G(\pi)} j = \binom{d+1}{2} - d - \sum_{j \in \asc_G(\pi)} j $.
    This means our $G$-major index polynomial is  
    \[
        \sum_{\pi \in S_d} q^{\maj_G(\pi)} = \sum_{\pi \in S_d} q^{\binom{d+1}{2} - d - \sum_{j \in \asc_G(\pi)} j}
    \]
    and by Corollary~\ref{cor:qznumeratorsymmetry}
    \[
        \sum_{\pi \in S_d} q^{\binom{d+1}{2} - d-\sum_{j \in \asc_G(\pi)} j} = \sum_{\pi \in S_d} q^{\binom{d}{2} - d -\sum_{j \in \asc_G(\pi)} d-j}.
    \]
    The highest power of this polynomial occurs when $ d +\sum_{j \in \asc_G(\pi)} d-j$ is minimized over all
$\pi \in S_d$ and the corresponding coefficient is given by the number permutations achieving this minimum. By Lemma~\ref{lem:minimumandcolorings} this value is given by $\min\left\{ \sum_{v \in V} f(v): f:V \to [m] \text{ a coloring} \right\}$. 
\end{proof}

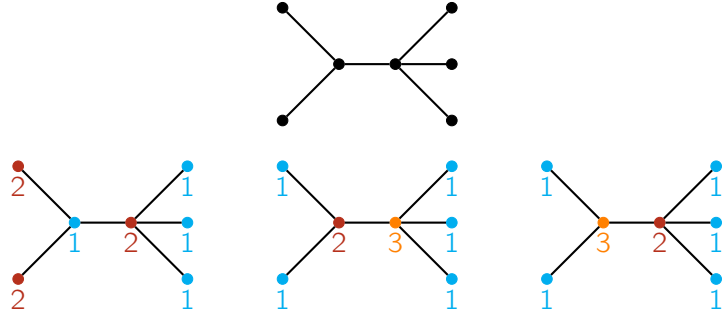
\begin{figure}
    \begin{tikzpicture}
    \begin{scope}[xshift=0, yshift=0, scale=0.75]    
        \vertex[fill](a1) at (-1,1) {};
        \vertex[fill](a2) at (-1,-1) {};
        \vertex[fill](a3) at (0,0) {};
        \vertex[fill](a4) at (1,0) {};
        \vertex[fill](a5) at (2,1) {};
        \vertex[fill](a6) at (2,0) {};
        \vertex[fill](a7) at (2,-1) {};
        

        \draw[thick] (a1) --  (a3);
        \draw[thick] (a2) --  (a3);
        \draw[thick] (a3) --  (a4);
        \draw[thick] (a4) --  (a5);
        \draw[thick] (a4) --  (a6);
        \draw[thick] (a4) --  (a7);
      
    \end{scope}

    \begin{scope}[xshift=0, yshift=-60, scale=0.75]    
        \vertex[fill,cyan](a1) at (-1,1) {};
        \vertex[fill,cyan](a2) at (-1,-1) {};
        \vertex[fill,BrickRed](a3) at (0,0) {};
        \vertex[fill,orange](a4) at (1,0) {};
        \vertex[fill,cyan](a5) at (2,1) {};
        \vertex[fill,cyan](a6) at (2,0) {};
        \vertex[fill,cyan](a7) at (2,-1) {};
        
        \node[anchor=north,cyan] at (a1) {$1$};
        \node[anchor=north,cyan] at (a2) {$1$};
        \node[anchor=north,BrickRed] at (a3) {$2$};
        \node[anchor=north,orange] at (a4) {$3$};
        \node[anchor=north,cyan] at (a5) {$1$};
        \node[anchor=north,cyan] at (a6) {$1$};
        \node[anchor=north,cyan] at (a7) {$1$};

        \draw[thick] (a1) --  (a3);
        \draw[thick] (a2) --  (a3);
        \draw[thick] (a3) --  (a4);
        \draw[thick] (a4) --  (a5);
        \draw[thick] (a4) --  (a6);
        \draw[thick] (a4) --  (a7);
      
    \end{scope}
    \begin{scope}[xshift=-100, yshift=-60, scale=0.75]    
        \vertex[fill,BrickRed](a1) at (-1,1) {};
        \vertex[fill,BrickRed](a2) at (-1,-1) {};
        \vertex[fill,cyan](a3) at (0,0) {};
        \vertex[fill,BrickRed](a4) at (1,0) {};
        \vertex[fill,cyan](a5) at (2,1) {};
        \vertex[fill,cyan](a6) at (2,0) {};
        \vertex[fill,cyan](a7) at (2,-1) {};
        
        \node[anchor=north,BrickRed] at (a1) {$2$};
        \node[anchor=north,BrickRed] at (a2) {$2$};
        \node[anchor=north,cyan] at (a3) {$1$};
        \node[anchor=north,BrickRed] at (a4) {$2$};
        \node[anchor=north,cyan] at (a5) {$1$};
        \node[anchor=north,cyan] at (a6) {$1$};
        \node[anchor=north,cyan] at (a7) {$1$};

        \draw[thick] (a1) --  (a3);
        \draw[thick] (a2) --  (a3);
        \draw[thick] (a3) --  (a4);
        \draw[thick] (a4) --  (a5);
        \draw[thick] (a4) --  (a6);
        \draw[thick] (a4) --  (a7);
      
    \end{scope}
    \begin{scope}[xshift=100, yshift=-60, scale=0.75]    
        \vertex[fill,cyan](a1) at (-1,1) {};
        \vertex[fill,cyan](a2) at (-1,-1) {};
        \vertex[fill,orange](a3) at (0,0) {};
        \vertex[fill,BrickRed](a4) at (1,0) {};
        \vertex[fill,cyan](a5) at (2,1) {};
        \vertex[fill,cyan](a6) at (2,0) {};
        \vertex[fill,cyan](a7) at (2,-1) {};
        
        \node[anchor=north,cyan] at (a1) {$1$};
        \node[anchor=north,cyan] at (a2) {$1$};
        \node[anchor=north,orange] at (a3) {$3$};
        \node[anchor=north,BrickRed] at (a4) {$2$};
        \node[anchor=north,cyan] at (a5) {$1$};
        \node[anchor=north,cyan] at (a6) {$1$};
        \node[anchor=north,cyan] at (a7) {$1$};

        \draw[thick] (a1) --  (a3);
        \draw[thick] (a2) --  (a3);
        \draw[thick] (a3) --  (a4);
        \draw[thick] (a4) --  (a5);
        \draw[thick] (a4) --  (a6);
        \draw[thick] (a4) --  (a7);
      
    \end{scope}
\end{tikzpicture}
    \caption{A graph and three colorings achieving the minimum sum of colors.}\label{fig:gmajorchromaticsum}
\end{figure}

\begin{example}
    The $G$-major index polynomial of the graph $G$ in Figure~\ref{fig:gmajorchromaticsum} is 
    \begin{align*}
        \sum_{\pi \in S_7} q^{\maj_G(\pi)} &= 3q^{18} + 8q^{17} + 20q^{16} + 56q^{15} + 101q^{14} + 167q^{13} + 249q^{12} + 358q^{11} + 448q^{10} \\
        &\quad + 529q^{9} + 555q^8 + 561q^7 + 552q^6 + 470q^5 + 358q^4 + 263q^3 + 170q^2 + 108q + 64.
    \end{align*}
    The leading coefficient is $3$ and so there are three colorings achieving the chromatic sum.
    These colorings must sum to
    \[
        \binom{8}{2} - 18 = 28-18 = 10.
    \]
    The three colorings are shown in Figure~\ref{fig:gmajorchromaticsum}.
    We note that this example also shows that the colorings achieving the minimum sum of colors for trees may require more than two colors.
\end{example}

Bajo et al.~\cite{qchromatic} made the following conjecture.

\begin{conjecture}[Bajo et al.]\label{conj:tildechitrees}
The leading coefficient of $\widetilde{\chi}_G^\mathbbm{1}(q,x)$ distinguishes trees.
\end{conjecture}

This motivates the following bounds on the degrees of the $G$-major polynomial for trees. 

\begin{corollary}\label{cor:treegmajordegree}
    Let $d \geq 2$. For a tree $T =([d],E)$, the degree of the $G$-major polynomial is bounded via
    \[
        \binom{d+1}{2} - \lfloor{1.5 d}\rfloor \leq 
        \mathrm{deg}\left(\sum_{\pi \in S_d} q^{\maj_G(\pi)}\right) \leq \binom{d+1}{2} - d -1 \, .
    \]
    Moreover, there exists a tree attaining each possible degree between the two bounds.
\end{corollary}
\begin{proof}
    We use a result of Kubicka and Schwenk~\cite[Theorem~3]{chromaticsums}.
    Let 
    \[
        \Sigma(T) := \min\left\{ \sum_{v \in V} f(v): f:V \to [m] \text{ a coloring} \right\}
    \]
    and $d := |V|$.
    We know every coloring of a tree has to use at least $2$ colors. 
    Then $\Sigma(T) \geq d+1$ via assigning one vertex the color $2$ and the rest the color $1$; a star graph attains this lower bound.
    Since trees are bipartite and thus $2$-colorable, we can color the vertices in the larger bipartite set with color $1$ and all the vertices in the smaller set as color $2$.
    So $\Sigma(T) \leq \ceil{\frac{d}{2}}+2 \floor{\frac{d}{2}} = \floor{1.5 d}$, and the path graph achieves this upper bound.
    Since $d+1 \leq \Sigma(T) \leq \floor{1.5 d}$,
    \[
        \binom{d+1}{2} - (d+1) \geq \binom{d+1}{2} - \Sigma(T) \geq \binom{d+1}{2} - \floor{1.5 d} \, ,
    \]
    Corollary~\ref{cor:gmajordegree} yields our first claim.
    
    To prove the second claim, it suffices to check there exists a tree $T$ with $\Sigma(T)=k$ for every $d+1 \leq k \leq \floor{1.5 d}$. 
    Let $B(m,n)$ denote the ``broom-like" tree where a path with $n$ vertices has $m$ leaves attached to one end, see Figure~\ref{fig:broomlikegraph} for an example. 
    Expressing $k=d+1+b$ for some $b$, it is not hard to
see that $\Sigma(B(d-2b-1,2b+1)) = k$.
\end{proof}

\begin{figure}
    \begin{tikzpicture}
    \begin{scope}[xshift=0, yshift=0, scale=1]    
        \vertex[fill](a1) at (0,0) {};
        \vertex[fill](a2) at (-1,1) {};
        \vertex[fill](a3) at (-1,0.333) {};
        \vertex[fill](a4) at (-1,-0.333) {};
        \vertex[fill] (a5) at (-1,-1) {};
        \vertex[fill] (a6) at (1,0) {};
        \vertex[fill](a7) at (2,0) {};
        

        \draw[thick] (a1) --  (a2);
        \draw[thick] (a1) --  (a3);
        \draw[thick] (a1) --  (a4);
        \draw[thick] (a1) --  (a5);
        \draw[thick] (a1) --  (a6);
        \draw[thick] (a6) --  (a7);
      
    \end{scope}
\end{tikzpicture}
    \caption{The ``broom-like" graph $B(4,3)$ as defined in the proof for Corollary~\ref{cor:treegmajordegree}.}\label{fig:broomlikegraph}
\end{figure}
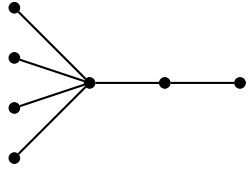

As a concluding aside for this section, observe that Stanley finds the degree of the polynomial
    \[
        W_P(q) : = \sum_{\sigma \in \jhset{P}} q^{\maj(\sigma)}
    \]
    for a naturally labeled poset $P$ as follows \cite[Section~3.15]{enumerativecombinatorics}.
    For $t \in P$, define $\delta(t)$ to be the number of vertices of the longest chain $t = t_1 < \cdots < t_{\delta(t)}$ of $P$ that starts at $t$. 
    We further define $\delta(P) = \sum_{t \in P} \delta(t)$.
    Stanley shows the degree of $W_P(q)$ equals $\binom{|P|+1}{2} - \delta(P)$, and $W_P(q)$ is monic, i.e., only one permutation achieves this maximal major index.
    We can connect this to our minimum sum coloring since we can compute the degree in two ways.
    For a graph $G$ on $[d]$ we can use Stanley's result and Theorem~\ref{thm:bijectiongstats} to compute the degree of $\sum_{\pi \in S_d} q^{\maj_G(\pi)} = \sum_{(\rho,\sigma)} q^{\maj(\sigma)}$ which is 
    \[
        \max_{\rho \in \acyclic{G}}\left\{ \binom{d+1}{2} - \delta(G_\rho) \right\} = \binom{d+1}{2} -
\min_{\rho \in \acyclic{G}} \left\{\delta(G_\rho) \right\}.
    \]
    By Corollary~\ref{cor:gmajordegree} this means
    \[
        \min_{\rho \in \acyclic{G}} \{\delta(G_\rho)\} = \min\left\{ \sum_{v \in V} f(v): f:V \to [m] \text{ a coloring} \right\}.
    \]
    For each acyclic orientation, this gives a candidate for the minimum sum coloring of graphs, namely a coloring $f:V(G) \to [k]$ defined by $f(t):=\delta(t)$ where $k$ is the number of elements of a longest chain in $G_\rho$.
    
    If we apply an alternative convention to maximize $\sum_{j \in \des(\sigma)} d-j$ over all $(\rho,\sigma)$, we can instead use the $G$-sequence colorings of permutations as described in the proof of Corollary~\ref{cor:gmajordegree} and Lemma~\ref{lem:minimumandcolorings}.
    This means that for each acyclic orientation, we have a candidate for the minimum sum coloring of graphs by assigning color based off the rank with respect to $G_\rho$ of an element $t \in G$, the number of vertices in the longest chain in $G_\rho$ ending at $t$.

\section{Symmetry}\label{sec:symmetry}

Theorem~\ref{thm:bijectiongstats} allows us to simplify some of the expressions arising in Section~\ref{sec:gstats} for the numerator polynomial of the $q$-chromatic generating function, showing that the coefficients of this numerator polynomial, each of which is a polynomials in $q$, are each palindromic up to a shift. 
We begin by introducing several important definitions required to describe the bijection established in Proposition~\ref{prop:ascentopposite} below.

Let $s:[d] \to [d]$ be the involution $s(i) := d+1 - i$.
For a graph $G$, an acyclic orientation $\rho \in \acyclic{G}$, and a natural labeling $\omega_\rho$ of $G_\rho$, we let $\rev{\rho}$ be the acyclic orientation obtained by reversing all the directions of $\rho$.
We then define
\[
\omega_\rho' :=s \circ \omega_{\rev{\rho}} \, .
\]
    Observe that $\omega_\rho'$ is a natural labeling. 

For each fixed $G$ and $\rho\in \acyclic{G}$, we next define a map sending a linear extension $\sigma \in \jhset{G_\rho,\omega_\rho}$ to $\op{\sigma}$, a linear extension of $G_{\rev{\rho}}$.
   Given $\sigma$, we know that there exists an order preserving map $\tau:G_\rho \to [d]$ such that $\sigma = \omega_\rho \circ \tau^{-1}$.
    We define 
    \[
    \op{\sigma}:= \omega_{\rev{\rho}}' \circ \tau^{-1} \circ s \, .
    \]
    In order to verify that $\op{\sigma} \in \jhset{G_{\rev{\rho}},\omega_{\rev{\rho}}'}$, we must check that $(\tau^{-1} \circ s)^{-1} = s \circ \tau$ is order preserving for $G_{\rev{\rho}}$. 
    If $i \leq j$ in $G_{\rev{\rho}}$, then $i \geq j$ in $G_\rho$ and so $\tau(i) \geq \tau(j)$, as $\tau$ is order preserving. Hence $d+1-\tau(i) \leq d+1 - \tau(j)$ which shows $s \circ \tau(i) \leq s \circ \tau(j)$. 

\begin{proposition}\label{prop:ascentopposite}
    Let $G$ be a graph on $[d]$ and for $\rho \in \acyclic{G}$, let $\omega_\rho$ be a natural labeling of $G_\rho$.
    There is a bijection 
    \[
        \psi: \{(\rho,\sigma) \mid \rho \in \mathcal{A}(G), \sigma \in \mathcal{L}(G_\rho,\omega_\rho) \} \to \{(\rho,\sigma) \mid \rho \in \mathcal{A}(G), \sigma \in \mathcal{L}(G_\rho,\omega_\rho') \}
        \]
        sending
        \[ (\rho,\sigma) \mapsto (\rev{\rho},\op{\sigma})
    \]
    with the property that $\asc(\sigma) = \{d-i: i \in \asc(\op{\sigma})\}$.
\end{proposition}
\begin{proof}   
    We claim $\psi$ is an involution and therefore also a bijection.
    Consider $\psi(\psi(\sigma,\rho)) =(\rev{(\rev{\rho})}, \op{(\op{\sigma})} )$.
    For $\rho \in \acyclic{G}$ we know $\rev{(\rev{\rho})} = \rho$, our labels for $G_\rho$ are $(\omega_\rho')' = s \circ s \circ \omega_\rho = \omega_\rho$, and
    \[
        \op{(\op{\sigma})}(i) = (\omega_\rho')' \circ (\tau^{-1} \circ s) \circ s = \omega_\rho \circ \tau^{-1} = \sigma
    \]
    and so $\psi$ is a bijection. 

    Moreover, the ascents under this bijection give our desired property. For $\rho \in \acyclic{G}$ and $\sigma \in \jhset{G_\rho,\omega_\rho}$,
    \begin{align*}
        j \in \asc(\sigma) &\Leftrightarrow \omega_\rho \circ \tau^{-1}(j) < \omega_\rho \circ \tau^{-1}(j+1) \\
        &\Leftrightarrow d+1 - \omega_\rho \circ \tau^{-1}(j) > d+1- \omega_\rho \circ \tau^{-1}(j+1) \\
        &\Leftrightarrow \omega_\rho' \circ \tau^{-1}(j) > \omega_\rho' \circ \tau^{-1}(j+1) \\
        &\Leftrightarrow \omega_\rho' \circ \tau^{-1} \circ s (d+1-j) > \omega_\rho' \circ \tau^{-1} \circ s (d-j) \\
        &\Leftrightarrow d-j \in \asc(\op{\sigma}) \, ,
    \end{align*} 
    hence $\asc(\sigma) = \{d-i: i \in \asc(\op{\sigma})\}$.
\end{proof}

Using Proposition~\ref{prop:ascentopposite}, we can simplify the expression of the numerator polynomial of the $q$-chromatic generating function for the linear form $\mathbbm{1} \in \mathbb{Z}^d$.
\begin{corollary}\label{cor:qznumeratorsymmetry}
    For $G$ a graph on $[d]$,
    \[
        \sum_{\pi \in S_d} q^{d + \sum_{j \in \asc_G(\pi)} d-j} z^{\ascnum_G(\pi) +1} = \sum_{\pi \in S_d} q^{d + \sum_{j \in \asc_G(\pi)} j} z^{\ascnum_G(\pi) +1} = \sum_{\pi \in S_d} q^{\binom{d+1}{2} - \maj_G(\pi)} z^{\ascnum_G(\pi) +1} .
    \]
\end{corollary}
\begin{proof}
    To prove the first equality of the corollary we have 
    \[
        \sum_{\pi \in S_d} q^{\sum_{j \in \asc_G(\pi)} d-j} z^{\ascnum_G(\pi)}= \sum_{(\rho,\sigma)} q^{\sum_{j \in \asc(\sigma)} d-j} z^{\ascnum(\sigma)}
    \] 
    by Corollary~\ref{cor:linearextensionnumerator}.
    Further applying Proposition~\ref{prop:ascentopposite} results in
    \[
        \sum_{(\rho,\sigma)} q^{\sum_{j \in \asc(\sigma)} d-j} z^{\ascnum(\sigma)} = \sum_{(\rho,\sigma)} q^{\sum_{d-j \in \asc(\op{\sigma})} d-j} z^{\ascnum(\op{\sigma})} = \sum_{(\rev{\rho},\op{\sigma})} q^{\sum_{i \in \asc({\sigma})} i} z^{\ascnum({\sigma})} = \sum_{(\rho,\sigma)} q^{\sum_{i \in \asc({\sigma})} i} z^{\ascnum({\sigma})}
    \]
    where the first equality holds by applying the permutation $\op{\sigma}$ and the last equality holds, as we have a bijection indexing the sum.
    Moreover any natural labeling of induced posets produces the same sum.
    We apply Theorem~\ref{thm:bijectiongstats}, so that 
    \[
        \sum_{(\rho,\sigma)} q^{\sum_{i \in \asc({\sigma})} i} z^{\ascnum({\sigma})} = \sum_{\pi \in S_d}
q^{\sum_{j \in \asc(\pi)} j} z^{\ascnum_G(\pi)} ,
    \]
    giving the first equality of the corollary.

    To derive the second equality, we know 
    \[
        \binom{d}{2} = \sum_{j \in [d-1]} j = \sum_{j \in \asc_G(\pi)} j + \sum_{i \in \des_G(\pi)} i  = \sum_{j \in \asc_G(\pi)} j + \maj_G(\pi)
    \]
    for any $\pi \in S_d$. This means
    \[
        d + \sum_{j \in \asc_G(\pi)} j = d + \binom{d}{2} - \maj_G(\pi) = \binom{d+1}{2} - \maj_G(\pi) \, ,
    \]
    which gives the desired equality.
\end{proof}

Another interesting consequence of Proposition~\ref{prop:ascentopposite} is the symmetry of the coefficients of the numerator polynomial of the $q$-chromatic generating function for the linear form $\mathbbm{1} \in \mathbb{Z}^d$.
\begin{corollary}\label{cor:palindromic}
   For a graph $G$ on $[d]$ with chromatic number $\xi$, write
    \[
        \sum_{\pi \in S_d} q^{d+\sum_{j \in \asc_G(\pi)} d-j} z^{\ascnum_G(\pi)+1} = a_\xi (q) z^\xi+ \cdots + a_d(q) z^d 
    \]
    with $a_i(q) := \sum q^{d+\sum_{j \in \asc_G(\pi)} d-j}$ where we sum over all permutations $\pi$ such that $\ascnum_G(\pi)+1 = i$. 
    For each $i \in [\xi,d]$, the polynomial $a_i(q)$ is shifted palindromic, i.e., $q^{d(i+1)} \, a_i( \frac 1 q) =a_i(q)$.
\end{corollary}
\begin{proof}
    For $i \in [\xi,d]$ consider $a_i(q)$. 
    By Corollary~\ref{cor:qznumeratorsymmetry}, we have $a_i(q) = \sum q^{d+\sum_{j \in \asc_G(\pi)} j}$ and so
    \begin{align*} 
        q^{d(i+1)} \, a_i \left( \frac 1 q\right) &= q^{d(i+1)} \sum_{\substack{\pi \in S_d \\ \asc_G(\pi)+1=i}} q^{-d-\sum_{j \in \asc_G} d-j } \\
        &= q^{d(\ascnum_G(\pi)+2)} \sum_{\substack{\pi \in S_d \\ \asc_G(\pi)+1=i}} q^{-d- d \cdot \ascnum_G(\pi) + \sum_{j \in \ascnum_G(\pi)} j }\\
        &=  \sum_{\substack{\pi \in S_d \\ \asc_G(\pi)+1=i}} q^{d\cdot\ascnum_G(\pi)+2d-d- d \cdot \ascnum_G(\pi) + \sum_{j \in \ascnum_G(\pi)} j }\\
        &= \sum_{\substack{\pi \in S_d \\ \asc_G(\pi)+1=i}} q^{d+\sum_{j \in \asc_G} j } \\
        &= a_i(q) \, .
    \end{align*}
    This means after factoring out the lowest power of $q$ in $a_i(q)$, we have a palindromic polynomial.
\end{proof}

\begin{example}
    One can see the palindromic polynomials in Example~\ref{ex:qchromaticgenfunction}.
    Moreover, the star graph in this example shows these polynomials may not necessarily be unimodal. 
\end{example}

    We remark further that for a fixed graph $G$ on $[d]$ and a fixed value $n \in \mathbb{Z}_{>0}$, the polynomial $\chi_G^\mathbbm{1}(q,n)$ is a
shifted palindromic polynomial in $q$.
    For every $n$-coloring $c$, we have an $n$-coloring $c' := n+1-c$ which is a bijection.
    Hence
    \[
        {q^{d(n+1)} \chi_G^\mathbbm{1}\left(\frac{1}{q},n\right) =\sum_{\substack{\text{\tiny colorings} \\ c:[d] \to [n]}} q^{\sum_{i \in [d]} n+1 -\sum_{i \in [d]} c(i)} = \sum_{\substack{\text{\tiny colorings} \\ c:[d] \to [n]}} q^{\sum_{i \in [d]} c'(i)} =  \sum_{\substack{\text{\tiny colorings} \\ c':[d] \to [n]}} q^{\sum_{i \in [d]} c(i)} =  \chi_G^\mathbbm{1}(q,n)}.
    \]

\bibliographystyle{amsplain}
\bibliography{ref}

\end{document}